\documentclass[10pt]{article}

\usepackage{url}
\usepackage{mathtools}
\usepackage{amssymb}
\usepackage{amsthm}
\usepackage{amsmath}
\usepackage{empheq}
\usepackage{latexsym}
\usepackage{enumerate}
\usepackage{eurosym}
\usepackage{dsfont}
\usepackage{appendix}
\usepackage{color} 
\usepackage[unicode]{hyperref}
\usepackage{frcursive}
\usepackage[utf8]{inputenc}
\usepackage[T1]{fontenc}
\usepackage{geometry}
\usepackage{multirow}
\usepackage{todonotes}
\usepackage{lmodern}
\usepackage{anyfontsize}
\usepackage{stmaryrd}
\usepackage{cleveref}
\usepackage{natbib}
\usepackage[english]{babel}
\usepackage[english=british]{csquotes}
\usepackage{tikz}
\usepackage{pgfplots}
\pgfplotsset{compat=1.16} 
\usepackage{float}

\setcitestyle{numbers,open={[},close={]}}

\definecolor{red}{rgb}{0.7,0.15,0.15}
\definecolor{green}{rgb}{0,0.5,0}
\definecolor{blue}{rgb}{0,0,0.7}
\hypersetup{colorlinks, linkcolor={red},citecolor={green}, urlcolor={blue}}
			
\makeatletter \@addtoreset{equation}{section}

\newtheorem{theorem}{Theorem}[section]
\newtheorem{assumption}[theorem]{Assumption}
\newtheorem{corollary}[theorem]{Corollary}

\newtheorem{lemma}[theorem]{Lemma}
\newtheorem{proposition}[theorem]{Proposition}

\newtheorem{definition}[theorem]{Definition}
\newtheorem{remark}[theorem]{Remark}

\DeclareUnicodeCharacter{014D}{\=o}
\def \E{\mathbb{E}}
\def \F{\mathbb{F}}

\def \H{\mathbb{H}}

\def \L{\mathbb{L}}

\def \N{\mathbb{N}}

\def \P{\mathbb{P}}

\def \R{\mathbb{R}}

\def \V{\mathbb{V}}

\def\Ac{{\cal A}}
\def\Bc{{\cal B}}
\def\Cc{{\cal C}}

\def\Ec{{\cal E}}
\def\Fc{{\cal F}}

\def\Hc{{\cal H}}

\def\Mc{{\cal M}}

\def\Rc{{\cal R}}

\def\Tc{{\cal T}}
\def\Uc{{\cal U}}
\def\Vc{{\cal V}}

\DeclareMathOperator*{\argmax}{arg\,max}

\DeclareMathOperator*{\esssup}{ess\,sup}

\title{Principal-Multiagents problem under equivalent changes of measure: general study and an existence result}

\author{Nicol\'as {\sc Hern\'andez-Santib\'a\~{n}ez} \footnote{CMM, Universidad de Chile, nhernandez@dim.uchile.cl.} }

        \date{\today}

\begin{document}

\maketitle

\begin{abstract}
We study a general contracting problem between the principal and a finite set of competitive agents, who perform equivalent changes of measure by controlling the drift of the output process and the compensator of its associated jump measure. In this setting, we generalize the dynamic programming approach developed by \citeauthor*{cvitanic2018dynamic} \cite{cvitanic2018dynamic} and we also relax their assumptions. We prove that the problem of the principal can be reformulated as a standard stochastic control problem in which she controls the continuation utility (or certainty equivalent) processes of the agents. Our assumptions and conditions on the admissible contracts are minimal to make our approach work. We review part of the literature and give examples on how they are usually satisfied. We also present a smoothness result for the value function of a risk--neutral principal when the agents have exponential utility functions. This leads, under some additional assumptions, to the existence of an optimal contract.

\vspace{5mm}

\vspace{5mm}


\end{abstract}

\section{Introduction}
The Principal-Agent problem was originally introduced in the 1970s and studied among others by \citeauthor*{mirrlees1976optimal} \cite{mirrlees1976optimal,mirrlees1999theory}, \citeauthor*{holmstrom1979moral} \cite{holmstrom1979moral}, \citeauthor*{shavell1979risk} \cite{shavell1979risk}, \citeauthor*{grossman1983analysis} \cite{grossman1983analysis}, \citeauthor*{rogerson1985first} \cite{rogerson1985first} and \citeauthor*{jewitt1988justifying} \cite{jewitt1988justifying}. In the last decade it has regained part of the importance it had in financial mathematics, mainly since the seminal paper by \citeauthor*{sannikov2008continuous} \cite{sannikov2008continuous} which proposed a new approach for the continuous--time version of the problem, that had been proposed long ago by \citeauthor*{holmstrom1987aggregation} \cite{holmstrom1987aggregation}.

\medskip
The dynamic programming approach, motivated by \citeauthor*{sannikov2008continuous} \cite{sannikov2008continuous} and formalized rigorously  by \citeauthor*{cvitanic2018dynamic} \cite{cvitanic2018dynamic}, gives the recipe to tackle the problem of the principal, which corresponds to a non-standard stochastic control problem as it involves the choice of a terminal payment %to the agent 
in the form of a general $\Fc_T$-measurable random variable. The trick is to use the continuation utility of the agent as an extra state variable for the problem of the principal, who can then reformulate her control problem by moving, without loss of generality, from the class of terminal payments to the class of terminal values of a controlled process. The resulting stochastic control problem is standard and can be associated to a Hamilton-Jacobi-Bellman (HJB) equation. Unfortunately, the assumptions in \cite{cvitanic2018dynamic} are very strong so the practical use of the paper is mainly to show the way to achieve such reformulation of the principal's problem. Many recent works in the literature cannot apply directly the main result in \cite{cvitanic2018dynamic} and have to mimic the proof by adjusting it to their own modelling assumptions. This is the case, for instance, of \citeauthor*{euch2021optimal} \cite{euch2021optimal}, \citeauthor*{baldacci2021optimal} \cite{baldacci2021optimal} or \citeauthor*{bensalem2020continuous} \cite{bensalem2020continuous}.

\medskip
The first contribution of the paper is to extend the main result of \cite{cvitanic2018dynamic}, in the context of equivalent changes of measures, by both considering a general setting and reducing the modelling assumptions. For the first one, we consider a principal who is signing a set of $N$ agents, each one of them is in charge of controlling an output process with jumps. For the second one, we propose a set of minimal assumptions for which our proof of the reformulation result holds. We review part of the literature on applications of the Principal-Agent model and we discuss how our assumptions are satisfied in those works.

\medskip
The study of the contracting problem with finitely many agents initiated in the 1980s with the works of \citeauthor*{holmstrom1982moral} \cite{holmstrom1982moral}, \citeauthor*{mookherjee1984optimal} \cite{mookherjee1984optimal}, \citeauthor*{green1983comparison} \cite{green1983comparison} and \citeauthor*{demski1984optimal} \cite{demski1984optimal} among others. Recently, \citeauthor*{elie2019contracting} \cite{elie2019contracting} tackled the continuous-time version of the Principal-Multiagent problem. Other works in continuous-time include the hierarchical model by \citeauthor*{hubert2020continuous} \cite{hubert2020continuous} and the problem with jumps by \citeauthor*{baldacci2021optimal} \cite{baldacci2021optimal}.

\medskip
The second contribution of the paper is to address the question of existence of the optimal contracts, by studying the reformulated problem of the principal, its associated HJB equation and the smoothness of its solution. In general, the stochastic problem of the principal is degenerated because she has the freedom to control the volatility of the continuation utility of the agent (a state variable of the problem). This results in a problem with a non-elliptic diffusion matrix and the HJB equation takes the form of a second-order fully non-linear parabolic PDE, for which the regularity results are not numerous. Without such regularity, the conclusions of the dynamic programming approach are incomplete. As remarked in \cite{cvitanic2018dynamic}, the optimal controls for the reformulated problem of the principal correspond to the partial derivatives of the value function $v$ and therefore the optimal contract (for the original problem) proposed by their approach needs at least some weak regularity of $v$ to be well-defined.

\medskip
In the setting of exponential utility of the agents, we outline a proof of the smoothness of the value function of a risk-neutral principal by studying the viscosity solutions of the HJB equation associated to her problem. 
As shown in the paper, in this specific case we can simplify the HJB equation considerably by reducing the number of state variables for the reformulated problem of the principal. The main idea is to use the certainty equivalent processes of the agents as state variables, whose dynamics lead to a straightforward optimization for the risk-neutral principal. Once this step is done, we characterize the value function $v$ as the unique viscosity solution of the simplified HJB equation, which also satisfies the properties to be associated to a forward-backwards SDE system (FBSDE). The smoothness of such system leads to the smoothness of its solution which corresponds again to the value function of the principal. 

\medskip
Under some additional assumptions, we can prove the existence of an optimal contract for the problem of the principal. The existence of an optimal contract is not an easy problem and there are few results in the literature dealing with it. Even in the one-dimensional setting of \citeauthor*{sannikov2008continuous} \cite{sannikov2008continuous}, the regularity of the value function is very hard to prove and it has been recently pointed out by \citeauthor*{possamai2020there} \cite{possamai2020there} that the extensive proof presented in \cite{sannikov2008continuous} has some gaps. Some works that follow (directly or indirectly) the dynamic programming approach and in which the existence of an optimal contract is proved include \citeauthor*{biais2010large} \cite{biais2010large}, \citeauthor*{pages2014mathematical} \cite{pages2014mathematical} and \citeauthor*{hernandez2021pollution} \cite{hernandez2021pollution}. Some existence results for the static Principal-Agent problem can be found in \citeauthor*{carlier2005existence} \cite{carlier2005existence} and \citeauthor*{moroni2014existence} \cite{moroni2014existence}, whose focus is placed on the monotone likelihood ratio condition (MLRC) and its consequences.

\medskip
The paper is structured as follows. The model is described in Section \ref{sec:model} as well as the game played by the agents and the problem of the principal. In Section \ref{sec:agent} we study the game played by the agents and present the reformulation of the problem of the principal. We provide the simpler results for the problem with a single agent in Section \ref{sec:single-agent}. In section \ref{sec:examples} we discuss how our result compares to the current literature. Finally, on Section \ref{sec:principal} we tackle the problem of existence of an optimal contract in the case of a risk-neutral principal and agents with exponential utility.

\paragraph*{Notations:} 
We let $\N$ be the set of integers, $\N^\star$ the set of positive integers and $\R_+$ the set of non-negative real numbers. For $d,n\in\N^\star$, $\mathcal{M}^{d,n}$ denotes the set of matrices with real entries, $d$ rows and $n$ columns. For $A\in\Mc^{d,n}$, $A^{i,:}$ denotes the $i-$th row of $A$. %For a stochastic process $X$, we denote by $\F^X:=(\Fc^X_t)_{t\geq 0}$ the filtration generated by $X$. 
For any $n\in\N^\star$, $i\in\{1,\dots,n\}$ and any vector $v\in\R^n$, we denote by $v^i$ the $i$-th coordinate of $v$ and by $v^{-i}$ the vector obtained by suppressing the $i$-th coordinate of $v$. For $u\in\R$, we denote by $u\otimes_i v$ the vector $w\in\R^{n+1}$ whose $i$-th coordinate is equal to $u$ and such that $w^{-i}=v$. We extend this operation to $u\in\R^\ell$, in which case $u\otimes_i v$ is the vector $w\in\R^{n+\ell}$ whose coordinates $i$-th to $(i+\ell)$-th are equal to $u$ and their removal results in $v$. We use the same notation for stochastic processes. $^\top$ denotes the transpose operation in $\R^N$.  For a function $v(t,x)$ we denote by $v_t$ its time partial derivative and $Dv$ its spatial gradient. For $k,\ell\in\N^\star$ and Euclidean spaces $\E_1$, $\E_2$, we denote by $C_b^{k,\ell}([0,T]\times\E_1;\E_2)$ the space of all functions $\varphi:[0,T]\times\E_1 \rightarrow \E_2$ which are $k$ times continuously differentiable in the first variable, $\ell$ times continuously differentiable in the second variable and the partial derivatives are uniformly bounded. We use the analogous notation for the set $C_b^{k}(\E_1;\E_2)$. We denote by $\Bc(\E_1;\E_2)$ the space of borel functions from $\E_1$ to $\E_2$. For a topological set $E$, we denote by $\Bc(E)$ the Borel sigma-algebra. For a stochastic process $X$, we denote the corresponding counting measure by $\mu_X$.

\section{The model} \label{sec:model}

We study the contracting problem of a principal who wants to hire $N$ agents, each one of them to manage a different outcome process $X^i$ with $i\in I:=\{1,\dots, N\}$. The contractual relationships between the principal and the agents are not independent since the action of each agent impacts the outcomes of the others. Therefore, there is competition between the agents. We study the problem in which the agents look for a Nash equilibrium of the non-cooperative game in which they are involved, while the principal plays a Stackelberg game against the whole set of agents.

\subsection{The setting} 
Let $(\Omega,\Fc,\P)$ be a probability space and $T>0$, so we consider a finite horizon. Every agent $i\in I$ is in charge of controlling a $d-$dimensional outcome process denoted $X^i$.  We let $W$ be an independent $n-$dimensional $\P-$Brownian motion. In the standard principal-agent problem, with $N=1$, the process $W$ is the source of randomness and unobservability of the actions of the agent, which leads to a contracting problem with moral hazard. In our framework, moral hazard is more complex since each process $X^i$ is affected by the Brownian motion as well as the actions of all the agents. We also include jumps in the dynamics of the outcome processes.

\medskip
For every $i\in I$ we let $\sigma^{i}: \Omega \times [0,T]\times \R^{dN} \longrightarrow \mathcal{M}^{d,n}$ be the covariance matrix for player $i$. Next, let $J^i$ be an independent multivariate point process with values in a Blackwell space $(E,\zeta)$. We denote by $\mu_{J^i}$ the counting measure on $([0,T]\times E,\Bc([0,T])\otimes\zeta)$ associated to each $J^i$ and its compensator under $\P$ is given by $F^i(\textrm{d}e)\textrm{d}t$. Here, $F^i$ is a positive $\sigma-$finite measure on $(E,\zeta)$ satisfying
$
\int_{E\setminus\{0\}} (1 \wedge |e|^2) F^i(\textrm{d}e).
$ %fixed nonatomic measure on $(E,\zeta)$ with infinite mass. 
Consider the coefficient $\beta^i:[0,T]\times\R^{dN}\times E\rightarrow\R^d$, which represents the sensitivity of $X^i$ on the jumps of $J^i$. Then, the dynamics of the process $X^i$, with $i\in I$, is given by
\begin{equation}\label{eq:dynamics-xi}
X^i_t = X_0^i +   \int_0^t  \sigma_s^{i}(X^1_s,\dots,X^N_s) \mathrm{d}W_s +   \int_0^t \int_{E\setminus\{0\}} \beta^i_s(X^1_s,\dots,X^N_s,e) \mu_{J^i}(\mathrm{d}s,\mathrm{d}e) , \quad t\in[0,T],
\end{equation}
where $X_0^i\in\R^d$ is given. We can write the dynamics of the output processes in vectorial form by defining $X:=((X^1)^\top,\dots,(X^N)^\top)^\top\in \R^{dN}$ and the matrix by blocks $\Sigma_s(X)\in \mathcal{M}^{dN,n}$, where for every $i\in I$ the row-block $i$ of $\Sigma_s(\cdot)$ corresponds to $\sigma_s^{i}(\cdot)$. We define similarly the $\R^N-$valued measure $\vec{\mu}_J:=(\mu_{J^1},\dots,\mu_{J^N})^\top$ and the function $\beta:[0,T]\times\R^{dN}\times E\rightarrow\Mc^{dN,N}$ by
\[
\beta_t(y,e):= %\text{diag}\big(\beta_t^1(y,e),\dots,\beta_t^N(y,e)\big) = 
\begin{pmatrix}
\beta_t^1(y,e) & 0 & \dots & 0 \\
0 &  \beta_t^2(y,e) & \dots & 0 \\
\vdots & \vdots & & \vdots \\
0 & 0 & \dots & \beta_t^N(y,e)
\end{pmatrix}
, \quad \forall t\in[0,t], ~y\in\R^{dN},~e\in E.
\]
%\[
%B_t: =\bigg( \int_0^t \int_{\R^d\setminus\{0\}} \beta^1_s(X^1_s,\dots,X^N_s,e) \mu_{J^1}(ds,de), \dots, \int_0^t \int_{\R^d\setminus\{0\}} \beta^N_s(X^1_s,\dots,X^N_s,e) \mu_{J^N}(ds,de)  \bigg)^\top.
%\]
Then we write equivalently the system of SDEs for the output processes as
\begin{equation}\label{eq:dynamics-vectorial}
X_t = X_0 + \int_0^t \Sigma_s(X_s) \mathrm{d}W_s + \int_0^t\int_{E\setminus\{0\}} \beta_s(X_s,e) \vec{\mu}_{J}(\mathrm{d}s,\mathrm{d}e), \quad t\in[0,T].
\end{equation}
We assume that for every $x\in \R^{dN}$ the process $\Sigma(\cdot,x)$ is $\F^W-$ predictable and both $\Sigma$ and $\beta$ are such that %\footnote{Standard conditions for this to hold are the Lipschitz-type conditions and linear growth. See for instance \cite[Theorem III.2.32]{jacod2003limit}.} 
the previous SDE has a unique strong solution, adapted to $\F^{W,J}$ the filtration generated by $W$ and $J^1,\dots,J^N$. We define the continuous martingale part of the output process by $X_t^c:=X_0 + \int_0^t \Sigma_s(X_s) \mathrm{d}W_s$. 

\medskip
For every $i\in I$, define the Borel transition kernel $\eta^{i}$ from $[0,T]\times\R^{dN}$ into $\R^d$, satisfying $\eta^i_t(y,\{0\})=0$, as follows 
%, by \cite[Theorem 14.53]{jacod2006calcul} or \cite[Remark III.2.28]{jacod2003limit} we can find a Borel function $\beta^i:[0,T]\times\R^{dN}\times E\rightarrow\R^d$ such that
\[
\eta^{i}_t(y,B) := \int_{E} {\bf 1}_{\beta^i_t(y,e)\in B\setminus\{0\}} F^i(\textrm{d}e), \quad \text{for any } B\in\Bc(\R^d), ~y\in\R^{dN}.
\]
Then, the compensator of each counting measure $\mu_{X^i}$ under $\P$ is given by (see for instance \cite[Theorem III.2.26]{jacod2003limit}) 
\[
\nu_t^i(\textrm{d}x)\textrm{d}t := \eta_t^i(X_t,\textrm{d}	x)\textrm{d}t.
\]

\begin{remark}
For technical reasons, we have assumed that two outcome processes controlled by different agents jump at the same time with probability zero. One could think of a more general model including multiple jump processes in each one of the dynamics as follows
\begin{equation}\label{eq:dynamics-xi-old}
X^i_t = X_0^i +   \int_0^t  \sigma_s^{i}(X^1_s,\dots,X^N_s) \mathrm{d}W_s +  \sum_{k=1}^{L_i}  \int_0^t \int_{E\setminus\{0\}} \beta^{i,k}_s(X^1_s,\dots,X^N_s,e) \mu_{J^{i,k}}(\mathrm{d}s,\mathrm{d}e) , \quad t\in[0,T],
\end{equation}
as long as none of the jump measures impacts more than one outcome. This is a much more complicated problem and, to the best of our knowledge, does not have applications in the literature other than when each $J^{i,k}$ is a point process. In some cases, the general dynamics \eqref{eq:dynamics-xi-old} can be rewritten in our form \eqref{eq:dynamics-xi}  by aggregating the measures $\mu_{J^{i,1}},\dots,\mu_{J^{i,L_i}}$ into a new measure $\mu^i$ on the space $E^{L_i}$. The main point here is whether or not the compensator of $\mu^i$ has the form $F^i(\emph{d}e)\emph{d}t$. This is indeed the case when each $J^{i,k}$ is a point process and in such setting the formulation with multiple jumps per outcome is equivalent to our formulation.
\end{remark}

Let $\F^X$ be the filtration generated by $(X^1,\dots,X^N)$ and $\F=(\Fc_t)_{t\geq 0}$ be the completion of $\F^X$ under $\P$. Let us introduce the following spaces
\begin{align*}
 \H^{2,1\times dN}_\text{loc}(X) & := \bigg\{  \Mc^{1,dN}\text{-valued, }\F\text{-predictable processes } Z: \int_0^T |Z_s\Sigma_s(X_s)|^2 ~ \mathrm{d}s <\infty, ~ \P-\text{a.s.}   \bigg\}, \\
 \L^{1}_{\text{loc}}(\mu_{X^i})  &:= \bigg\{    \R\text{-valued, }\F\text{-predictable function } H:  \sum_{s\leq \cdot}  \bigg(\int_{\R^d\setminus\{0\}} H_s(x)\mu_{X^i}(\mathrm{d}t,\mathrm{d}x) \bigg)^\frac{1}{2} \text{ is }\P\text{-loc. int.} \bigg\}, ~ \text{for } i\in I.
\end{align*}

We assume that $X$ satisfies the following martingale representation property by components.

\begin{assumption}\label{ass:MRP-coordinates}
For every $(\F,\P)-$martingale $M$ there exists $Z\in\H_{loc}^{2,1\times dN}(X)$ and $\{H^i\}_{i=1}^N$, with $H^i\in\L_{loc}^1(\mu_{X^i})$, such that it holds $\P-$a.s.
\[
\emph{d}M_t = Z_t \emph{d}X_t^c + \sum_{i=1}^N \int_{\R^d\setminus\{0\}} H_t^i(x) \big( \mu_{X^i}(\emph{d}t,\emph{d}x)-\nu_t^i(\emph{d}x)\emph{d}t\big), \quad t\in[0,T].
\]
\end{assumption}

\begin{remark}
The martingale representation property is a standard assumption for the Principal-Agent model and, to the best of our knowledge, does not rule out any application of the model in the literature. However, the representation of martingales takes the following form
\[
\emph{d}M_t = Z_t \emph{d}X_t^c + \int_{\R^{dN}\setminus\{0\}} H_t(x) \big( \mu_{X}(\emph{d}t,\emph{d}x)-\nu_t(\emph{d}x)\emph{d}t\big),  \quad t\in[0,T],
\]
where $\nu_t$ is the compensator of $\mu_X$. Therefore, our main assumption is that the representation can be done directly with respect to the measures $\mu_{X^1},\dots,\mu_{X^N}$. The representation with respect to a family of measures is not a trivial question, as discussed by \citeauthor*{jacod1977general} in \cite[Remark 2]{jacod1977general}. It is proved by \citeauthor*{euch2021optimal} \cite{euch2021optimal} that it holds true if every $J_i$ is a point process. If $X$ is a pure-jump process, the case with a single jump measure is proved by \citeauthor*{jacod2003limit} \cite[Theorem III.1.26]{jacod2003limit} and  \cite[Theorem III.4.29]{jacod2003limit}, whereas the general case was tackled recently by \citeauthor*{calzolari2021martingale} \cite{calzolari2021martingale}. In Appendix \ref{app:MRP-coordinates}, we show how to obtain in our setting the representation with respect to all the components of $X$ when neither of the explosion times are finite.
\end{remark}

\subsection{The weak formulation} \label{sec-weak-formulation}

As usual in contract theory, we assume that the actions of the agents affect the distributions of the output processes. In our competitive setting each agent affects all the outputs. We have thus a weak formulation of the problem that we present next.

\medskip
For every $i\in I$, let $A_i$ be the finite dimensional set of actions of agent $i$. We write $A=\prod_{i=1}^N A_i$ and $A^{-i}=\prod_{j=1, j\neq i}^N A_j$. We consider the coefficient functions $b:\Omega \times [0,T]\times \R^{dN} \times A \longrightarrow\R^n$, %{\color{blue}bounded} and 
$\lambda^i:\Omega \times [0,T]\times \R^{dN} \times A \times E \longrightarrow(0,+\infty)$, such that for every $(y,a,e)\in \R^{dN} \times A\times E$ the processes $b(\cdot,y,a)$ and $\lambda^i(\cdot,y,a,e)$ are $\F-$predictable and $\lambda^i(\cdot,y,a,\cdot)$ is locally integrable over $[0,T]\times E$. The set of admissible actions of the agents is defined as follows.

\begin{definition} \label{def:admissible-controls}
We say that an $A-$valued, $\F-$predictable process $\alpha$ is an admissible joint action, if the following process is a $\P-$martingale 
\[
M_t^\alpha: = \Ec\left( \int_0^t b_s(X_s,\alpha_s) \cdot \mathrm{d}W_s + \sum_{i=1}^N \int_0^t \int_{E\setminus\{0\}} \big( \lambda_s^i(X_s,\alpha_s,e) - 1 \big) \big( \mu_{J^i}(\mathrm{d}s,\mathrm{d}e) - F^i(\mathrm{d}e)\mathrm{d}s \big)  \right).
\]
We define $\Ac$ as the set of admissible joint actions.
\end{definition}
%\begin{definition} \label{def:admissible-controls}
%We say that an $A-$valued, $\F-$predictable process $\alpha$ is an admissible joint action, if 
%\[
%\E^{\P} \bigg( e^{\frac12 \int_0^T ||b_s(X_s,\alpha_s)||^2 \mathrm{d}s + \int_0^T \int_{E\setminus\{0\}} \big( \lambda_s^i(X_s,\alpha_s,x) - 1 \big)^2 F^i(\mathrm{d}x)\mathrm{d}s }  \bigg) < \infty.
%\]
%We define $\Ac$ as the set of admissible joint actions.
%\end{definition}
%For every $\alpha\in\Ac$, define
%\[
%M_t^\alpha: = \Ec\left( \int_0^t b_s(X_s,\alpha_s) \cdot \mathrm{d}W_s + \sum_{i=1}^N \int_0^t \int_{E\setminus\{0\}} \big( \lambda_s^i(X_s,\alpha_s,x) - 1 \big) \big( \mu_{J^i}(\mathrm{d}s,\mathrm{d}x) - F^i(\mathrm{d}x)\mathrm{d}s \big)  \right),
%\]
%and note that the process $M^\alpha$ is a positive $\P$-martingale.
Note that $M^\alpha$ is a positive $\P$-martingale.\footnote{From \cite[Remark 15.3.1.]{cohen2015stochastic}, if we write $M_t^\alpha=\Ec(\tilde M_t^\alpha)$, we have that $M^\alpha$ is positive if and only if $\Delta \tilde M^\alpha >-1$, which is the case since each $\lambda^k$ is positive.} We define the probability measure $\P^\alpha$ through $
\frac{d\P^\alpha}{d\P} \big |_{\Fc_T} = M_T^\alpha .
$
It follows from Girsanov's theorem that the process $W^\alpha:= W - \int_0^\cdot b_s(X_s,\alpha_s) \mathrm{d}s$ is a $\P^\alpha-$Brownian motion and $X$ has the dynamics
\begin{equation}\label{eq:dynamics-vectorial-alpha}
X_t = X_0 + \int_0^t \Sigma_s(X_s)b_s(X_s,\alpha_s) \mathrm{d}s + \int_0^t \Sigma_s(X_s) \mathrm{d}W_s^\alpha +  \int_0^t\int_{E\setminus\{0\}} \beta_s(X_s,e) \vec{\mu}_{J}(\textrm{d}s,\textrm{d}e) , \quad t\in[0,T].
\end{equation}

We will denote by $X_t^{c,\alpha}:= \int_0^t \Sigma_s(X_s) \mathrm{d}W_s^\alpha$ the continuous martingale part of $X$ under $\P^\alpha$. Note that, from \cite[Theorem 15.3.10]{cohen2015stochastic}, the compensator of each counting measure $\mu_{X^i}$ under $\P^\alpha$ is given by 
\[
\nu_t^{i,\alpha}(\textrm{d}x)\textrm{d}t := \eta_t^{i,\alpha}(X_t,\textrm{d}	x)\textrm{d}t, 
\]
where
\[
 \eta^{i,\alpha}_t(y,B) = \int_{E} {\bf 1}_{ \beta^i_t(y,e)\in B\setminus\{0\}} \lambda_t^i(y,\alpha_t,e) F^i(\textrm{d}e), \quad \text{for any } B\in\Bc(\R^d), ~y\in\R^{dN}.
\]
We define the vector of measures $\vec{\mu}_X=(\mu_{X^1},\dots,\mu_{X^N})^\top$ and for $\alpha\in\Ac$, we let $\vec{\nu}_t^\alpha=(\nu_t^{1,\alpha},\dots,\nu_t^{N,\alpha})^\top$ be the transition kernel of its compensator under the measure $\P^\alpha$.
\begin{remark}
As can be seen from equation \eqref{eq:dynamics-vectorial-alpha}, the drift of the process $X$ under $\P^\alpha$ belongs to the range of the matrix $\Sigma$ which restricts the modeling choices of the problem when $\Sigma$ is degenerate. Loosely speaking, if the matrix $\Sigma\Sigma^\top$ is invertible, then the range of $\Sigma$ is the whole space $\R^{dN}$ and there is no restriction at all. However, for the sake of generality, we do not impose this assumption and we deal with an eventual degenerate matrix $\Sigma$.
\end{remark}

\begin{remark} \label{rem:lambda-independent}
Note that if the function $\lambda^i$ is independent of $e$, then we have  $\eta^{i,\alpha}_t(y,B) = \lambda_t^i(y,\alpha_t) \eta_t^i(y,B)$. This is the usual setting in the literature, in which the agents are impacting the compensator of the measures in a multiplicative way.
\end{remark}

\subsection{The game between the agents}

Each agent is hired by the principal to control the distribution of the output process he is in charge of, and his actions have an impact on the projects of the other agents. The agents are competitive and therefore they play an $N-$player differential game for which we assume they look for a Nash equilibrium.

\medskip
Fix an arbitrary agent $i\in I$. His actions are costly subject to the function $c^i:\Omega\times [0,T] \times \R^{dN} \times A \longrightarrow\R_+$, which is %{\color{blue} bounded} %, which is {\color{blue}non-decreasing and convex} in $a$ and 
such that $c^i(\cdot,x,a)$ is $\F-$predictable for every $(x,a)\in \R^{dN} \times A$. For his work the agent receives a terminal remuneration $\xi^i \in\Rc$ and continuous payments $\chi^i=(\chi^i_t)_{t\in[0,T]}\in\Xi$, which belong to the sets
\[
\Rc := \{ \xi^i : \R\text{-valued, } \Fc_T\text{-measurable r.v.} \}, \quad \Xi := \{ \chi^i: \R_+\text{-valued, } \F\text{-predictable process} \}.
\]
A pair $(\xi^i,\chi^i)\in\Rc\times\Xi$ is referred to as a contract. %We %define also $\Rc:=\prod_{i=1}^N \Rc^i$, $\Xi:=\prod_{i=1}^N \Xi^i$ and 
When there is no confusion we also call contracts the elements $(\xi,\chi)\in\Rc^N\times\Xi^N$ which are offered to the whole group of agents. In the next subsection we will impose some integrability conditions, that we omit for now, to define the set of admissible contracts. 

\medskip
The agent $i$ has a quasi-separable utility function, as shown below, where the function $\Uc_A^i:\R \longrightarrow\R$ is continuous, increasing and $u_A^i:\R^N \longrightarrow\R$ is measurable%and bounded}
. He discounts the future at instantaneous rate $\rho^i:\Omega\times [0,T] \times \R^{dN} \times\R^N \times A \longrightarrow\R$, which is a continuous map such that $\rho^i(\cdot,x,k,a)$ is $\F-$predictable for every $(x,k,a)\in \R^{dN}\times\R^N \times A$. 

\medskip
Given a contract $(\xi,\chi)\in\Rc^N\times\Xi^N$ offered by the principal to the agents, if they perform the joint action $\alpha\in\Ac$, the utility obtained by agent $i$ is given by 
\[
U_0^i(\alpha^i,\alpha^{-i},\xi^i,\chi^i):=  \E^{\P^\alpha} \left[ e^{-\int_0^T \rho^i_s(X_s,\chi_s,\alpha_s) \mathrm{d}s} ~ \Uc_A^i (\xi^i) + \int_0^T e^{-\int_0^s \rho^i_u(X_u,\chi_u,\alpha_u) \mathrm{d}u} \big( u_A^i(\chi_s) - c^i_s(X_s,\alpha_s) \big) \mathrm{d}s  \right].
\]
The best response of agent $i$ to the actions $\alpha^{-i}$ of the others is obtained by solving the following problem
\begin{equation} \label{eq:agent-problem}
V_0^i(\alpha^{-i},\xi^i,\chi^i):= \sup_{\alpha^i\in\Ac^i(\alpha^{-i})} \E^{\P^\alpha} \left[ e^{-\int_0^T \rho^i_s(X_s,\chi_s,\alpha_s) \mathrm{d}s} ~ \Uc_A^i (\xi^i) + \int_0^T e^{-\int_0^s \rho^i_u(X_u,\chi_u,\alpha_u) \mathrm{d}u} \big( u_A^i(\chi_s) - c^i_s(X_s,\alpha_s) \big) \mathrm{d}s  \right],
\end{equation}
where the set of controls for player $i$ is 
\[
\Ac^i(\alpha^{-i}) : = \big\{ \alpha^i : A_i\text{-valued, } \F\text{-predictable process}, \text{ such that } \alpha^i \otimes_i \alpha^{-i} \in \Ac \big\}.
\]

\begin{remark}
Note that the instantaneous utility and cost functions, as well as the discount factors, are assumed to depend on the whole vector of actions and payments of the agents. One can reduce the dependence of these functions to the particular action and payment of the corresponding agent if such setting is better suited for an application of the model. However, to keep the model as general as possible, we refrain to do so and point out that the techniques developed to approach both settings are exactly the same.
\end{remark}

\begin{definition}\label{def:nash}
We say the joint action $\alpha^\star\in\Ac$ is a Nash equilibrium for the contract $(\xi,\chi)\in\Rc^N\times\Xi^N$, denoted by $\alpha^\star\in\emph{NE}(\xi,\chi)$, if for every $i\in I$
\[
V_0^i(\alpha^{\star,-i},\xi^i,\chi^i) = U_0^i(\alpha^{\star,i},\alpha^{\star,-i},\xi^i,\chi^i).
\]
\end{definition}

The intuition is that, given a contract $(\xi,\chi)\in\R^N\times\Xi^N$ offered by the principal, the agents will perform a joint action from the set $\text{NE}(\xi,\chi)$. This action can be recommended by the principal herself and none of the agents would have incentives to deviate from his individual recommendation, if he assumes the others are following their recommendations. The agents will enter into a contractual relationship with the principal if the value they obtain from the Nash equilibrium is greater than their reservation values, denoted by $R_0^i\in\R$, for $i\in I$. We assume the reservation values are exogenous and known by the principal.

\begin{remark}
Our setting extends naturally the standard contracting problem with a single agent, which corresponds to the case $N=1$. In such case, the set of Nash equilibria $\emph{NE}(\xi,\chi)$ reduces to the set of optimal efforts of the single agent, that we denote later by $\Ac^\star(\xi,\chi)$ (see Section \ref{sec:single-agent}).
\end{remark}

\begin{remark}
We aim to cover the most well known Principal-Agent models in the literature. We extend the one-dimensional multi-agent problem by \citeauthor{elie2019contracting} \cite{elie2019contracting} for which we also relax the assumptions on the functions of the model and include jumps. If we take $N=1$ and $\beta^1=0$, by choosing $\rho\equiv 0$ we obtain the finite horizon version of Sannikov's model \cite{sannikov2008continuous}. By choosing $\Uc_A(x) = -\exp(-R_A x)$, with $R_A>0$, and $u_A \equiv c\equiv 0$ we recover the model by Holmstr\"om and Milgrom \cite{holmstrom1987aggregation} where $\rho$ can be chosen, for instance, with the form $\rho_s(x,k,a)=R_A(\hat{u}_A(k)-\hat{c}_s(x,a))$. This is the main reason for adding into $\rho$ the dependence on the continuous payments. If $N=1$ and $\beta^1(y,e)=e$, we obtain the model by \citeauthor*{capponi2015dynamic} \cite{capponi2015dynamic}. In the equivalent formulation \eqref{eq:dynamics-xi-old}, if $N=1$, $L_1=2$, $\beta^{1,1}(y,e)=1$, $\beta^{1,2}(y,e)=-1$, $F^1$ and $F^2$ are Dirac masses at $\{1\}$, we recover the model by \citeauthor*{euch2021optimal} \cite{euch2021optimal}. %\textcolor{red}{revisar porque se integra un $P$ c/r al Poisson}.
\end{remark}

Lastly, we introduce the dynamic version of the best response of agent $i$, to the actions of the others. This family of random variables is important for defining the set of admissible contracts of the principal in the next subsection. Given a contract $(\xi^i,\chi^i)\in\Rc\times\Xi$ for agent $i$, and actions $\alpha^{-i}$ of the other agents, we define for $\tau\in\Tc_{0,T}$\footnote{$\Tc_{0,T}$ denotes the set of stopping times with values in $[0,T]$ and $\E^{\P^\alpha}_\tau$ denotes the conditional expectation under $\P^\alpha$ with respect to $\Fc_\tau$.} 
\begin{equation}\label{eq:dynamic-value-agent}
%V^i(\tau,\alpha^{-i},\xi^i,\chi^i) := \esssup_{\alpha^i\in \Ac^i(\alpha^{-i})} \E^{\P^\alpha} \left[ e^{-\int_\tau^T \rho^i_s(X_s,\chi_s,\alpha_s) \mathrm{d}s} ~ \Uc^i_A (\xi^i) + \int_\tau^T e^{-\int_\tau^s \rho^i_u(X_u,\chi_u,\alpha_u) \mathrm{d}u} \big( u_A^i(\chi_s) - c^i_s(X_s,\alpha_s) \big) \mathrm{d}s ~ \bigg| ~ \Fc_\tau  \right],
V^i(\tau,\alpha^{-i},\xi^i,\chi^i) := \esssup_{\alpha^i\in \Ac^i(\alpha^{-i})} \E_\tau^{\P^\alpha} \left[ e^{-\int_\tau^T \rho^i_s(X_s,\chi_s,\alpha_s) \mathrm{d}s} ~ \Uc^i_A (\xi^i) + \int_\tau^T e^{-\int_\tau^s \rho^i_u(X_u,\chi_u,\alpha_u) \mathrm{d}u} \big( u_A^i(\chi_s) - c^i_s(X_s,\alpha_s) \big) \mathrm{d}s   \right],
\end{equation}
which represents the value obtained by the agent during the time interval $[\tau,T]$.

\subsection{Principal's problem}

The principal hires all the agents and wants to maximize her expected utility by anticipating the Nash equilibria resulting for the contracts. She has a terminal utility function $\Uc_P:\R^N \longrightarrow\R$ %, which is {\color{blue} measurable%continuous, increasing and concave
%},
and an instantaneous utility function $u_P:\R^N \longrightarrow\R$, both functions being measurable. She liquidates the output processes through the measurable map $L:\R^{dN}\longrightarrow\R$ %, which is {\color{blue} measurable%bounded increasing
%}, 
and discounts the future at instantaneous rate $r:\Omega\times [0,T] \times \R^{dN} \longrightarrow\R_+$, such that $r(\cdot,x)$ is $\F-$predictable for every $x\in \R^{dN}$. We start by defining the set of admissible contracts, that the principal can offer to the agents.

\begin{definition}\label{def-admisible-contracts}
Let $\xi=(\xi^1,\dots,\xi^N)\in\Rc^N$, $\chi=(\chi^1,\dots,\chi^N)\in\Xi^N$. We say that  $(\xi,\chi)$ is an admissible contract if it satisfies the following conditions
\begin{enumerate}[(i)]
\item $\forall \alpha\in\Ac$,  $\exists p_\alpha >1$ such that  $\sup_{\tau\in\Tc_{0,T}}  \E^{\P^\alpha} \left[  \big| V^i(\tau,\alpha^{-i},\xi^i,\chi^i)\big|^{p_\alpha} \right] < \infty$, for every $i\in I$.
\item $\forall \alpha\in\Ac$,  $\exists \hat p_\alpha >1$ such that $\frac{1}{p_\alpha} + \frac{1}{\hat p_\alpha}<1$ and 
\[
\sup_{\tau\in\Tc_{0,T}}  \E^{\P^\alpha} \left[ e^{-\hat p_\alpha \int_0^\tau \rho^i_s(X_s,\chi_s,\alpha_s) \mathrm{d}s} \right] < \infty, \quad \text{for every } i\in I.
\]
\item For every $\alpha\in\Ac$, $\tau\in\Tc_{0,T}$ we have
\[
\E^{\P^\alpha} \left[ \bigg| \int_0^\tau e^{-\int_0^s \rho^i_u(X_u,\chi_u,\alpha_u) \mathrm{d}u} \big( u_A^i(\chi_s) - c^i_s(X_s,\alpha_s) \big) \mathrm{d}s \bigg| \right] < \infty, \quad \text{for every } i\in I.
\]
\item There exists $\alpha^\star\in \emph{NE}(\xi,\chi)\neq \emptyset$ such that $V_0^i(\alpha^{\star,-i},\xi^i,\chi^i) \geq R_0^i$ for every $i\in I$. %$V_A(\xi,\chi) \in [R_0,\infty)$.
%\item $\E^{\P^\alpha}   \left[ e^{-\int_0^T r_s(X_s) \mathrm{d}s } ~ | \Uc_P \big(L(X_T) - \sum_{i=1}^N \xi^i \big) | + \int_0^T e^{-\int_0^s r_s(X_u) \mathrm{d}u} | u_P(\chi_s)| \mathrm{d}s \right] < \infty$, for every $\alpha\in\text{NE}(\xi,\chi)$.
\end{enumerate}
We denote by $\Cc$ the set of admissible contracts.
\end{definition}

Let us start the discussion on the admissibility conditions by saying that point $(iv)$ contains the usual assumptions in the Principal-Agent literature. First, the game played by the agents will have Nash equilibria that the principal can anticipate, so then given any contract, she can recommend the agents the best joint action process to perform. %\footnote{We assume that the agents will follow the recommendation of the principal has more than one optimal control, he will choose the one that benefits the principal the most.} 
Second, the agents will accept the contract and follow the recommended actions since it corresponds to a Nash equilibrium and their expected utility under it are greater than their reservation values.

\medskip
The other conditions are required for the problems involved to be well-defined, and to satisfy some technical properties that will be used in the following proofs. They are also chosen as the minimal required assumptions. Note that if the coefficient function $b$ is bounded and each $\lambda^i$ is bounded by an $F^i-$integrable function, then the density measures $(M^\alpha)_{\alpha\in\Ac}$ have finite moments of any order. Consequently, conditions $(i)-(iii)$ reduce to the same integrability properties but only under the initial measure $\P$.

\medskip
We conclude the section by introducing the problem of the principal, who offers to the agents an admissible contract which maximizes her own expected utility, subject to the result of their game
\begin{equation}\label{eq:principal-problem}
V_P:= \sup_{(\xi,\chi)\in\Cc} ~  \sup_{\alpha \in\text{NE}(\xi,\chi)} ~  \E^{\P^{\alpha}} \left[ e^{-\int_0^T r_s(X_s) \mathrm{d}s } ~ \Uc_P \bigg(L(X_T) - \sum_{i=1}^N \xi^i \bigg)  - \int_0^T e^{-\int_0^s r_s(X_u) \mathrm{d}u}  u_P(\chi_s) \mathrm{d}s \right].
\end{equation}

\section{Solving the agents' game}\label{sec:agent}

In this section, we fix a contract $(\xi,\chi)\in\Cc$ offered by the principal and we characterize its set of Nash equilibria, that is, the elements of $\text{NE}(\xi,\chi)$. We show that the principal can keep track of the actions of the agents by using their continuation values as state variables, which allows to reformulate her optimization problem as a standard stochastic control problem. We give a special treatment to the case of exponential utility, as we show that in such setting the principal can use different state variables, namely the certainty equivalents of the agents.  All the proofs are deferred to Appendix \ref{app:proofs-agents}.

\subsection{General case}

%In this setting the map $\rho$ is intended to represent a discount factor, that we assume bounded.\footnote{This assumption is not enforced in the next subsection}
%\begin{assumption}\label{ass-rho}
%The map $\rho:\Omega\times [0,T] \times \R^d \times\R\times A \longrightarrow\R$ is bounded.
%\end{assumption}
Let us introduce the function $f:\Omega\times[0,T]\times\R^{dN}\times\R^N\times\Mc^{N,dN} \times \Bc(\R^d;\Mc^{N,N}) \times\R^N\times A \longrightarrow\R^N$ defined component-wise\footnote{We use the convention that if any $h^{i,\ell}$ is not integrable with respect to $\eta_s^{a,\ell}(x,\textrm{d}u)$ then $f^i=-\infty$.} by
\[
f^{i}_s(x,y,z,h,k,a) = u^i_A(k) -c^i_s(x,a) - 	\rho^i_s(x,k,a) y^i +  z^{i,:}\Sigma_s(x)b_s(x,a) + \sum_{\ell=1}^N \int_{\R^d\setminus\{0\}} h^{i,\ell}(u) \eta_s^{\ell,a}(x,\textrm{d}u) , \quad i\in I.
\]
Our first assumption in this section is that the map $f$ possesses solutions to a fixed-point type of equation, which is linked to the Nash equilibria of the contract.
\begin{assumption}\label{ass-unique-max-f}
For every $(s,x,y,z,h,k)\in [0,T]\times\R^{dN}\times\R^N\times\Mc^{N,dN}\times \Bc(\R^d;\Mc^{N,N}) \times\R^N$ there exists a unique element in $\Omega\times A$, denoted by $a^\star(s,x,y,z,h,k)$, such that for any $i\in I$ it holds $\P-$a.s.
\[
\big\{ a^{\star,i}(s,x,y,z,h,k) \big\} = \argmax_{a^i\in A_i} ~ f_s^i(x,y,z,h,k,a^i\otimes_i a^{\star,-i}(s,x,y,z,h,k)).
\]
\end{assumption}

\begin{remark}
If the map $f^i$ is separable in $a\in A$ then the previous condition can be decoupled and it reduces to the existence of a maximizer of $f^i$. This is the case, for instance, if each $\lambda^i$ is independent of $e$ (see Remark \ref{rem:lambda-independent}) and the maps $\lambda^i$, $c^i$, $\rho^i$ and $b$ are linear in $a$ (or depend only on $a^i$).
\end{remark}

We define next the function $F:\Omega\times[0,T]\times\R^{dN}\times\R^N\times\Mc^{N,dN}\times \Bc(\R^d;\Mc^{N,N})\times\R^N\longrightarrow\R^N$ by
\[
F_s(x,y,z,h,k) := f_s(x,y,z,h,k,a^\star(s,x,y,z,h,k)).
\]

Consider now the following multidimensional BSDE with jumps, where $\Uc_A(\xi)\in\R^N$ denotes the vector whose $i-$th coordinate is equal to $\Uc^i_A(\xi^i)$.
\begin{equation}\label{eq:bsde-agent}
Y_t = \Uc_A(\xi) + \int_t^T F_s(X_s,Y_s,Z_s,H_s,\chi_s) \mathrm{d}s - \int_t^T Z_s \mathrm{d}X^c_s - \int_t^T \int_{\R^d\setminus\{0\}} H_t(x) \vec{\mu}_X(\textrm dt, \textrm dx) , \quad t\in[0,T].
\end{equation}

\begin{definition}
A solution to the BSDE \eqref{eq:bsde-agent} is a triple $(Y,Z,H)$ such that $Y$ is an $\R^N-$valued $\F-$semimartingale satisfying \eqref{eq:bsde-agent} and $(Z,H)\in\V(X):= \H^{2,N\times dN}_\text{loc}(X)\times \L^{1,N\times N}_{\text{loc}}(\vec{\mu}_X)$, where
\begin{align*}
& \H^{2,N\times dN}_\text{loc}(X)  = \bigg\{  \Mc^{N,dN}\text{-valued, }\F\text{-predictable processes } Z: \int_0^T |Z_s\Sigma_s(X_s)|^2 ~ \mathrm{d}s <\infty, ~ \P-\text{a.s.}   \bigg\}, \\
& \L^{1,N\times N}_{\text{loc}}(\vec{\mu}_X)  = \bigg\{    \Mc^{N,N}\text{-valued, }\F\text{-predictable function } H:  H^{i,\ell}\in \L^1_\text{loc}(\mu_{X^\ell})  ~~ \forall (i,\ell)\in I\times I  \bigg\}.
\end{align*}

\end{definition}

Our second and last assumption in this section will be that the map $F$ provides well-possedness of the forward version of BSDE \eqref{eq:bsde-agent}, which is used to obtain a convenient representation of the set of admissible contracts.

\begin{assumption}\label{ass:nice-F}
The map $F$ is such that %: $(i)$ F
for every $y\in\R^N$ and $(Z,H)\in \V(X)$ there exists a unique strong solution to the following SDE
\[
Y^{y,Z,H,\chi}_t = y - \int_0^t F_s(X_s,Y^{y,Z,H,\chi}_s,Z_s,H_s,\chi_s) \emph{d}s + \int_0^t Z_s\emph{d}X^c_s + \int_0^t \int_{\R^d\setminus\{0\}} H_t(x) \vec{\mu}_X(\mathrm{d}t, \mathrm{d}x) , \quad t\in[0,T]. 
\]
%$(ii)$ For every $y_1,y_2\in\R^N$ and $Z\in\H^{2,N\times dN}_\text{loc}(X)$ there exists a constant $C$ such that for every $t\in[0,T]$
%\[
%|Y^{y_1,Z,\chi,i}_t - Y^{y_2,Z,\chi,i}_t | \leq C |y_1^i-y_2^i|, \quad i\in I.
%\]
\end{assumption}

\begin{remark}
A sufficient condition for Assumption \ref{ass:nice-F} is that the map $F$ is Lipschitz on $y$. In the case of a single-agent, this is the case for instance if the map $\rho$ is bounded (see Remark \ref{rem-rho-bounded}).
\end{remark}

\medskip
In Proposition \ref{prop-agent-bsde} we show that the processes with the previous form correspond to the continuation values of the agents under the Nash equilibrium of the contract. At time $t=0$ those values are given by $y\in\R^N$ which assures the acceptance of the contracts by the agents when they are bigger than the reservation values. To establish the correspondence, we must impose some conditions to be satisfied by the processes $Z$ and $H$. We define now the class of processes to which we will restrict our attention, as a function of the initial values of the agents $y$.

\begin{definition} (i) For $y\in\R^N$ and $(Z,H)\in \V(X)$, we define the process $a^{\star,y,Z,H,\chi}$ by
\[
a_s^{\star,y,Z,H,\chi} := a^{\star}(s,X_s,Y^{y,Z,H,\chi}_s,Z_s,H_s,\chi_s), \quad \emph{d}t\otimes \emph{d}\P-\text{a.s. over } [0,T]\times \Omega.
\]
(ii) For $y\in\R^N$, we denote by $\Vc^{y,\chi}$ the following class of processes
%\begin{align*}
%\Zc^{y,\chi}:= \bigg\{   Z\in\H^{2,N\times dN}_\text{loc}(X):  \text{ the process } a^{\star,y,Z,\chi} \in\Ac; ~ \forall i\in I, Y_T^{y,Z,\chi,i} \in \mathrm{Im}(\Uc_A^i), \P-\mathrm{a.s.}  \text{ and }  \\
%\forall i\in I, \forall \alpha^i\in\Ac^i( a^{\star,y,Z,\chi,-i}), ~ \exists q_{\alpha^i} >1 \text{ such that } ~~~~&  \\ 
% \frac{1}{q_{\alpha^i}} + \frac{1}{\hat p_{\alpha^i \otimes_i a^{\star,Z,\chi,-i}}} < 1, ~~  \sup_{\tau\in\Tc_{0,T}} \E^{\P^{\alpha^i\otimes_i a^{\star,y,Z,\chi,-i}}}[| Y_\tau^{y,Z,\chi,i} |^{q_{\alpha^i}}] < \infty \bigg\}. & 
%\end{align*}

\begin{align*}
\Vc^{y,\chi}:= \bigg\{   (Z,H)\in\V(X):  a^{\star,y,Z,H,\chi} \in\Ac; ~ \forall i\in I, ~ Y_T^{y,Z,H,\chi,i} \in \mathrm{Im}(\Uc_A^i) ~ \P-\mathrm{a.s.}  \text{ and }   \forall \alpha^i\in\Ac^i( a^{\star,y,Z,H,\chi,-i}), ~ ~~~~&  \\ 
\exists q_{\alpha^i} >1  \text{ s.t. } \frac{1}{q_{\alpha^i}} + \frac{1}{\hat p_{\alpha^i \otimes_i a^{\star,Z,H,\chi,-i}}} < 1, ~~  \sup_{\tau\in\Tc_{0,T}} \E^{\P^{\alpha^i\otimes_i a^{\star,y,Z,H,\chi,-i}}}[| Y_\tau^{y,Z,H,\chi,i} |^{q_{\alpha^i}}] < \infty \bigg\}. & 
\end{align*}
\end{definition}

We present now the main result of this section. As usual in the literature, we establish an equivalence between finding the Nash equilibria to the contract $(\xi,\chi)$ and solving BSDE \eqref{eq:bsde-agent}. %which represents the value of the agent as the solution to this BSDE.

\begin{proposition}\label{prop-agent-bsde}
For every $\alpha^\star\in\emph{NE}(\xi,\chi)$ there exists a solution $(Y,Z,H)$ to BSDE \eqref{eq:bsde-agent}, with $(Z,H)\in\Vc^{Y_0,\chi}$, such that 
\begin{equation}\label{eq:a-star}
\alpha_s^{\star,i} = a_s^{\star,Y,Z,H,\chi,i}, \quad \emph{d}t\otimes \emph{d}\P-\text{a.s. over } [0,T]\times \Omega, \quad \forall i\in I.
\end{equation}
Conversely, let $(Y,Z,H)$ be a solution to BSDE \eqref{eq:bsde-agent}, with $(Z,H)\in\Vc^{Y_0,\chi}$. Then the control defined by \eqref{eq:a-star} belongs to $\emph{NE}(\xi,\chi)$.
\end{proposition}

\medskip
The previous proposition allows us to reformulate the problem of the principal as a standard stochastic control problem, by restricting without loss of generality the form of the admissible contracts. The unique Nash equilibrium to the contracts in the class $\Vc^{\chi,y}$ follows immediately from Proposition \ref{prop-agent-bsde}.

\begin{corollary}\label{cor-optimal-response}
Let $y\in\R^N$, $(Z,H)\in\Vc^{y,\chi}$ and each $\xi^i= (\Uc_A^i)^{-1}( Y_T^{y,Z,H,\chi,i})$. Then the joint action $a^{\star,y,Z,H,\chi}$ belongs to $\emph{NE}(\xi,\chi)$ and the values of the agents are given by $V_0^i(\alpha^{\star,-i},\xi^i,\chi^i)=y^i$, for every $i\in I$.
\end{corollary}

\medskip
By combining all the results in this section, we obtain the reformulation of the problem of the principal. 
%\begin{equation*} \label{eq:reformulated-principal}
%V_P = \sup_{y \geq R_0, \chi\in\Xi_2} \sup_{(Z,H)\in\Zc^{y,\chi}} ~  \E^{\P^{a^{\star,y,Z,H,\chi}}} \left[ e^{-\int_0^T r_s(X_s) \mathrm{d}s } ~ \Uc_P \bigg(L(X_T) - \sum_{i=1}^N \Uc_A^{-1}(Y_T^{y,Z,H,\chi,i}) \bigg) - \int_0^T e^{-\int_0^s r_s(X_u) \mathrm{d}u}  u_P(\chi_s) \mathrm{d}s \right],
%\end{equation*}
\begin{align*} \label{eq:reformulated-principal}
V_P = \sup_{y \geq R_0, \chi\in\Xi_2} \sup_{(Z,H)\in\Vc^{y,\chi}} ~  \E^{\P^{a^{\star,y,Z,H,\chi}}} \bigg[  e^{-\int_0^T r_s(X_s) \mathrm{d}s } ~  \Uc_P \bigg(L(X_T) &- \sum_{i=1}^N \Uc_A^{-1}(Y_T^{y,Z,H,\chi,i}) \bigg) \\
 & - \int_0^T e^{-\int_0^s r_s(X_u) \mathrm{d}u}  u_P(\chi_s) \mathrm{d}s \bigg],
\end{align*}
where $\Xi_2$ is the set of $\F-$predictable processes $(\chi_t)_{t\in[0,T]}$ such that $\Vc^{y,\chi}\neq\emptyset$ for some $y$, $R_0\in\R^N$ is the vector whose $i-$th coordinate is $R_0^i$ and the corresponding inequality is component-wise.

\subsection{The case of exponential utility} \label{sec:exp-utility}

Assume now that for every agent $i\in I$, we have $u_A^i, c^i \equiv 0$ and $\Uc_A^i(x)=-\exp(-R_A^i x)$, with $R_A^i > 0$ the risk aversion of agent $i$. In the case of exponential utility the map $\rho^i$ usually takes the form $\rho^i_s(x,k,a)=R_A^i(\hat{u}_A^i(s,x,k)-\hat{c}_s^i(x,a))$, for some utility and cost functions $\hat u_A^i$ and $\hat c^i$. However, for the sake of generality we do not impose $\rho^i$ to have this specific form. 
%and we just assume it is function bounded by below.
%\begin{assumption}\label{ass-rho-cert-equiv}
%The map $\rho:\Omega\times [0,T] \times \R^d \times\R\times A \longrightarrow\R$ is bounded by below.
%\end{assumption}

\medskip
In this setting we can obtain an alternative representation of the admissible contracts, by using the certainty equivalent processes of the agents instead of their continuation values. To this end, we introduce the corresponding function $g:\Omega\times[0,T]\times\R^{dN}\times\Mc^{N,dN}\times  \Bc(\R^d;\Mc^{N,N})\times\R^N \times A \longrightarrow\R^N$ defined component-wise by 
\[
g^i_s(x,z,h,k,a) :=  \frac{1}{R_A^i}\rho^i_s(x,k,a) + z^{i,:}\Sigma_s(x)b_s(x,a) - \frac{1}{2}R_A^i ||z^{i,:}\Sigma_s(x)||^2 + \frac{1}{R_A^i} \sum_{\ell=1}^N  \int_{\R^d}\big(1-e^{R_A^i h^{i,\ell}(u)})\eta_s^{\ell,a}(x,\textrm{d}u),~  i\in I.
\]
As in the previous subsection, we assume that the map $g$ possesses solutions to the corresponding fixed-point equation.
\begin{assumption}\label{ass-unique-max-g}
For every $(s,x,z,h,k)\in [0,T]\times\R^{dN}\times\Mc^{N,dN}\times  \Bc(\R^d;\Mc^{N,N}) \times\R^N$ there exists a unique element in $\Omega\times A$, denoted by $\hat a^\star(s,x,z,h,k)$, such that for any $i\in I$ it holds $\P-$a.s.
\[
\big\{ \hat a^{\star,i}(s,x,z,h,k) \big\} = \argmax_{a^i\in A_i} ~ g_s^i(x,z,h,k,a^i\otimes_i \hat a^{\star,-i}(s,x,z,h,k)).
\]
\end{assumption}

We define now $G:\Omega\times[0,T]\times\R^{dN}\times\Mc^{N,dN}\times \Bc(\R^d;\Mc^{N,N})\times\R^N \longrightarrow\R^N$ by
\[
G^i_s(x,z,h,k) :=  g^i_s(x,z,h,k,\hat a^\star(s,x,z,h,k)).
\]
Consider next the following BSDE with jumps, with the vector $\xi:=(\xi^1,\dots,\xi^N)^\top$
\begin{equation}\label{eq:bsde-agent-cert-equiv}
Y_t = \xi + \int_t^T G_s(X_s,Z_s,H_s,\chi_s) \mathrm{d}s - \int_t^T Z_s \mathrm{d}X^c_s + \int_t^T \int_{\R^d\setminus\{0\}} H_s(x)\vec{\mu}_X(\textrm{d}t,\textrm{d}x) , \quad t\in[0,T].
\end{equation}
For $y\in\R^N$ and $(Z,H)\in\V(X)$, we define the process $Y^{y,Z,H,\chi}$ as follows\footnote{Note this is a direct definition since $Y^{y,Z,H,\chi}$ does not appear in the right-hand side.}
\[
Y^{y,Z,H,\chi}_t := y - \int_0^t G_s(X_s,Z_s,H_s,\chi_s) \textrm{d}s + \int_0^t Z_s \mathrm{d}X^c_s - \int_0^t \int_{\R^d\setminus\{0\}} H_s(x)\vec{\mu}_X(\textrm{d}t,\textrm{d}x) , \quad t\in[0,T]. 
\]
We define now the class of processes to which we will restrict our attention. In the previous section we reformulated the problem of the principal through classes of processes associated to fixed values of the agent. In the setting of exponential utility, this reformulation can be achieved through a single class of processes independent of the initial values.

%\begin{definition} We denote by $\hat \Zc^\chi$ the following class of processes%, with the notation $\hat a_s^\star := \hat a^\star(s,X_s,Z_s,\chi_s)$
%% \[
%%\hat \Zc^\chi:= \bigg\{   Z\in\H^{2,d}_\text{loc}(X):  \Ec\bigg(-R_A \int_0^\cdot Z_s\cdot dX_s^{c,\hat a^\star}\bigg)  \text{ is a }\P^{\hat a^\star}-\text{martingale, } \sup_{\tau\in\Tc_{0,T}} \E[|\Uc_A(Y_\tau^{0,Z,\chi})|^p] < \infty \text{ for some } p>1 \bigg\}. 
%%\]
% \[
%\hat \Zc^\chi:= \bigg\{   Z\in\H^{2,d}_\text{loc}(X): \forall \alpha\in\Ac, ~\exists q_\alpha >1  \text{ such that } \frac{1}{q_\alpha} + \frac{1}{\hat p_\alpha} < 1 \text{ and }    \sup_{\tau\in\Tc_{0,T}} \E^{\P^\alpha}[|\Uc_A(Y_\tau^{0,Z,\chi})|^{q_\alpha}] < \infty  \bigg\}. 
%\]
%\end{definition}
\begin{definition} We denote by $\hat \Vc^\chi$ the following class of processes, with $a^{\star,Z,H,\chi}_s: = \hat a^\star(s,X_s,Z_s,H_s,\chi_s)$
\begin{align*}
\hat\Vc^\chi:= \bigg\{   (Z,H)\in\V(X):  \text{ the process } a^{\star,Z,H,\chi} \in\Ac \text{ and }  \forall i\in I, ~ \forall \alpha^i\in\Ac^i( a^{\star,Z,H,\chi,-i}), ~ \exists q_{\alpha^i} >1 \text{ such that } ~~&  \\ 
 \frac{1}{q_{\alpha^i}} + \frac{1}{\hat p_{\alpha^i \otimes_i a^{\star,Z,H,\chi,-i}}} < 1, ~~  \sup_{\tau\in\Tc_{0,T}} \E^{\P^{\alpha^i\otimes_i a^{\star,Z,H,\chi,-i}}}[|\Uc_A^i(Y_\tau^{0,Z,H,\chi,i})|^{q_{\alpha^i}}] < \infty \bigg\}. & 
\end{align*}
\end{definition}

%\begin{definition} We denote by $\hat \Zc^\chi$ the class of processes $Z\in\H^{2,d}_\text{loc}(X)$ satisfying  \\
%$(i)$ $\forall \alpha\in\Ac, ~\exists q_\alpha >1$ such that $\frac{1}{q_\alpha} + \frac{1}{\hat p_\alpha} < 1$ and  
%\[
% \sup_{\tau\in\Tc_{0,T}} \E^{\P^{\alpha\otimes\alpha^{\star,-i} }}[|\Uc_A(Y_\tau^{0,Z,\chi})|^{q_\alpha}] < \infty.
% \] 
%$(ii)$ The process $a^{\star,Z,\chi} \in\Ac$, where
%\[
%a^{\star,Z,\chi}_s := \hat a^\star(s,X_s,Z_s,\chi_s), \quad \emph{d}t\otimes\emph{d}\P-\text{a.s. over } [0,T]\times \Omega.
%\]
%\end{definition}

We have then an alternative representation of the set $\text{NE}(\xi,\chi)$, which will be used in this setting. The following is an analogous to Proposition \ref{prop-agent-bsde}. %which represents the value of the agent as the solution to this BSDE.

\begin{proposition}\label{prop-agent-bsde-cert-equiv}
For every $\alpha^\star\in \emph{NE}(\xi,\chi)$ there exists a solution $(Y,Z,H)$ to BSDE \eqref{eq:bsde-agent-cert-equiv}, with $(Z,H)\in\hat\Vc^\chi$, such that
\begin{equation}\label{eq:a-star-cert-equiv}
\alpha^\star_s = a^{\star,Z,H,\chi}_s , \quad \emph{d}t\otimes\emph{d}\P-\text{a.s. over } [0,T]\times \Omega.
\end{equation}
Conversely, let $(Y,Z,H)$ be a solution to BSDE \eqref{eq:bsde-agent-cert-equiv}, with $(Z,H)\in\hat\Vc^\chi$. Then the control $a^{\star,Z,H,\chi}$ belongs to $\emph{NE}(\xi,\chi)$.
\end{proposition}

We present next the analogous result to Corollary \ref{cor-optimal-response}.

\begin{corollary}\label{cor-optimal-response-cert-equiv}
Let $y\in\R^N$, $(Z,H)\in\hat\Vc^\chi$ and each remuneration $\xi^i=Y_T^{y,Z,H,\chi,i}$. Then the Nash equilibrium of the game is given by $a^{\star,Z,H,\chi}$ and the value of each agent under it is $V_0^i(\alpha^{\star,-i},\xi^i,\chi^i)=\Uc_A^i(y^i)$.
%\[
%\alpha^\star_s = \hat a^\star(s,X_s,Z_s), \quad \emph{d}t\otimes\emph{d}\P-\text{a.s. over } [0,T]\times \Omega.
%\]
\end{corollary}

To conclude this section, we obtain the reformulation of the problem of the principal. 
\begin{equation} \label{eq:reformulated-principal-cert-equiv}
V_P = \sup_{ y \geq \Uc_A^{-1}(R_0) } \sup_{\chi\in\hat\Xi_2} \sup_{(Z,H)\in\hat\Vc^{\chi}} ~  \E^{\P^{a^{\star,Z,H,\chi}}} \left[ e^{-\int_0^T r_s(X_s) \mathrm{d}s } ~ \Uc_P \bigg(L(X_T) - \sum_{i=1}^N Y_T^{y,Z,H,\chi,i} \bigg) - \int_0^T e^{-\int_0^s r_s(X_u) \mathrm{d}u}  u_P(\chi_s) \mathrm{d}s \right],
\end{equation}
where $\hat\Xi_2$ is the set of $\F-$predictable processes $(\chi_t)_{t\in[0,T]}$ such that $\hat\Vc^{\chi}\neq\emptyset$, the vector $\Uc_A^{-1}(R_0)\in\R^N$ is the one whose $i-$th coordinate is $(\Uc_A^i)^{-1}(R_0^i)$ and the corresponding inequality is component-wise.

\subsection{The case of a single agent} \label{sec:single-agent}

The purpose of this subsection is to present in an easy-to-read form the results in the setting of a single agent, that is $N=1$. In the case of a single agent, the model is much more simple and there is no need to look for Nash equilibria of a stochastic differential game since there is no competition. All the results are, of course, an immediate consequence of the ones from the previous subsections. We start by presenting the standard definitions in this setting.

\medskip
When $N=1$, the weak formulation of the problem follows directly from the one developed in \Cref{sec-weak-formulation}. For a volatility function $\sigma: \Omega \times [0,T]\times \R^d \longrightarrow \mathcal{M}^{d,n}$, Borel function $\beta:[0,T]\times\R^{d}\times E\rightarrow\R^d$ and the drift coefficient $b:\Omega \times [0,T]\times \R^d \times A \longrightarrow\R^n$, satisfying the same assumptions as before, the set of admissible controls $\Ac$ is given by the one containing the $\F-$predictable processes $\alpha$ with values in a finite dimensional set $A$, such that the following process is a $\P-$martingale
\[
M_t^\alpha: = \Ec\left( \int_0^t b_s(X_s,\alpha_s) \cdot \mathrm{d}W_s  +  \int_0^t \int_{E\setminus\{0\}} \big( \lambda_s(X_s,\alpha_s,e) - 1 \big) \big( \mu_{J}(\mathrm{d}s,\mathrm{d}e) - F(\mathrm{d}e)\mathrm{d}s \big)  \right).
\]
It follows from Girsanov's theorem that the controlled process $X$ has the dynamics
\[
X_t = X_0 +  \int_0^t \sigma_s(X_s) b_s(X_s,\alpha_s) \mathrm{d}s + \int_0^t \sigma_s(X_s) \mathrm{d}W^\alpha_s +  \int_0^t \int_{E\setminus\{0\}} \beta_s(X_s,e) \vec{\mu}_{J}(\textrm{d}s,\textrm{d}e),
\]
where $W^\alpha$ is a $\P^\alpha-$Brownian motion. % Moreover, the compensator of the counting measure $\mu_{X}$ under $\P^\alpha$ is given by 
%\[
%\nu_t^{\alpha}(\textrm{d}x)\textrm{d}t := \sum_{i=1}^N  \eta_t^{i,\alpha}(X_t,\textrm{d}	x)\textrm{d}t, 
%\]
%where
%\[
% \eta^{i,\alpha}_t(y,B) = \int_{E} {\bf 1}_{\beta^i_t(y,e)\in B\setminus\{0\}} \lambda_t^i(y,\alpha_t,e) F^i(\textrm{d}e), \quad \text{for any } B\in\Bc(\R^d), ~y\in\R^{d}.
%\]
Given a contract $(\xi,\chi)$ offered by the principal, the problem of the agent is
\begin{equation} \label{eq:single-agent-problem}
V_A(\xi,\chi):= \sup_{\alpha\in\Ac} \E^{\P^\alpha} \left[ e^{-\int_0^T \rho_s(X_s,\chi_s,\alpha_s) \mathrm{d}s} ~ \Uc_A (\xi) + \int_0^T e^{-\int_0^s \rho_u(X_u,\chi_u,\alpha_u) \mathrm{d}u} \big( u_A(\chi_s) - c_s(X_s,\alpha_s) \big) \mathrm{d}s  \right].
\end{equation}

We denote by $\Ac^\star(\xi,\chi)$ the set of solutions to problem \eqref{eq:single-agent-problem}, that is the set of optimal efforts for the agent. As usual, the agent will enter into a contractual relationship with the principal if the value obtained from the contract is greater than his reservation value, denoted by $R_0$. We introduce also the dynamic version of the value of the agent, which represents the value obtained during the time interval $[\tau,T]$ and allows to define the set of admissible contracts for the principal
\begin{equation*} \label{eq:dynamic-value-single-agent}
V(\tau,\xi,\chi) := \esssup_{\alpha\in\Ac} \E^{\P^\alpha} \left[ e^{-\int_\tau^T \rho_s(X_s,\chi_s,\alpha_s) \mathrm{d}s} ~ \Uc_A (\xi) + \int_\tau^T e^{-\int_\tau^s \rho_u(X_u,\chi_u,\alpha_u) \mathrm{d}u} \big( u_A(\chi_s) - c_s(X_s,\alpha_s) \big) \mathrm{d}s ~ \bigg| ~ \Fc_\tau  \right].
\end{equation*}

\begin{definition}\label{def-admisible-contracts-single-agent}
Let $\xi$ be an $\Fc_T-$measurable random variable and $(\chi_t)_{t\in[0,T]}$ an $\F-$predictable process. We say that the pair $(\xi,\chi)$ is an admissible contract if it satisfies the following conditions
\begin{enumerate}[(i)]
\item $\forall \alpha\in\Ac$,  $\exists p_\alpha >1$ such that  $\sup_{\tau\in\Tc_{0,T}}  \E^{\P^\alpha} \left[  \big| V(\tau,\xi,\chi) \big|^{p_\alpha} \right] < \infty$.
\item $\forall \alpha\in\Ac$,  $\exists \hat p_\alpha >1$ such that $\frac{1}{p_\alpha} + \frac{1}{\hat p_\alpha}<1$ and 
\[
\sup_{\tau\in\Tc_{0,T}}  \E^{\P^\alpha} \left[ \int_0^\tau e^{-\hat p_\alpha \int_0^s \rho_u(X_u,\chi_u,\alpha_u) \mathrm{d}u} \right] < \infty.
\]
\item For every $\alpha\in\Ac$, $\tau\in\Tc_{0,T}$ we have
\[
\E^{\P^\alpha} \left[ \bigg| \int_0^\tau e^{-\int_0^s \rho_u(X_u,\chi_u,\alpha_u) \mathrm{d}u} \big( u_A(\chi_s) - c_s(X_s,\alpha_s) \big) \mathrm{d}s \bigg| \right] < \infty.
\]
\item The set  $\Ac^\star(\xi,\chi)\neq \emptyset$ and the value $V_A(\xi,\chi) \geq R_0$. %$V_A(\xi,\chi) \in [R_0,\infty)$.
%\item $\E^{\P^\alpha}   \left[ e^{-\int_0^T r_s(X_s) \mathrm{d}s } ~ | \Uc_P \big(L(X_T) - \xi \big) | + \int_0^T e^{-\int_0^s r_s(X_u) \mathrm{d}u} | u_P(\chi_s)| \mathrm{d}s \right] < \infty$, for every $\alpha\in\Ac^\star(\xi,\chi)$.
\end{enumerate}
We denote by $\Xi$ the set of admissible contracts.
\end{definition}

\medskip
The problem of the principal, who offers to the agent an admissible contract which maximizes her own expected utility subject to his optimal action, takes the form.
\begin{equation}\label{eq:principal-problem-single-agent}
V_P:= \sup_{(\xi,\chi)\in\Xi} ~  \sup_{\alpha \in\Ac^\star(\xi,\chi)} ~  \E^{\P^{\alpha}} \left[ e^{-\int_0^T r_s(X_s) \mathrm{d}s } ~ \Uc_P \big(L(X_T) - \xi \big)  - \int_0^T e^{-\int_0^s r_s(X_u) \mathrm{d}u}  u_P(\chi_s) \mathrm{d}s \right].
\end{equation}

\subsubsection{General case}

%In this setting the map $\rho$ is intended to represent a discount factor, that we assume bounded.\footnote{This assumption is not enforced in the next subsection}
%\begin{assumption}\label{ass-rho}
%The map $\rho:\Omega\times [0,T] \times \R^d \times\R\times A \longrightarrow\R$ is bounded.
%\end{assumption}

Let us introduce the functions $F:\Omega\times[0,T]\times\R^d\times\R\times\R^d\times\Bc(\R^d;\R)\times\R\longrightarrow\R$ and $f:\Omega\times[0,T]\times\R^d\times\R\times\R^d\times\Bc(\R^d;\R)\times\R\times A \longrightarrow\R$
\begin{align*}
 f_s(x,y,z,h,k,a) &= u_A(k) -c_s(x,a) - 	\rho_s(x,k,a) y + z^\top\sigma_s(x)b_s(x,a) + \int_{\R^d\setminus\{0\}} h(u) \eta_s^a(x,\textrm{d}u), \\
 F_s(x,y,z,h,k) &= \sup_{a\in A} f_s(x,y,z,k,a).
\end{align*}
Our first assumption in this section, will be that the function $f$ is maximized at a unique point.
\begin{assumption}\label{ass-unique-max-f-single-agent}
For every $(s,x,y,z,h,k)\in [0,T]\times\R^d\times\R \times\R^d \times\Bc(\R^d;\R)\times\R$ the map $a\mapsto f_s(x,y,z,h,k,a)$ has $\P-$a.s. a unique maximizer over the set $A$, denoted by $a^\star(s,x,y,z,h,k)$.
\end{assumption}

\begin{remark}
Existence of a maximizer can be ensured if $b$, $c$, $\rho$ and $\lambda$ are continuous in $a$ and the set $A$ is compact, or alternatively if the map $-f$ is coercitive in $a$. Uniqueness of the maximizer is obtained if $f$ is strictly concave in $a$, which is the case for instance if $b$, $\rho$ and $\lambda$ are linear and $c$ is strictly convex in $a$.
\end{remark}

Consider next the following BSDE, which is linked to the solutions to the problem of the agent.
\begin{equation}\label{eq:bsde-single-agent}
Y_t = \Uc_A(\xi) + \int_t^T F_s(X_s,Y_s,Z_s,H_s,\chi_s) \mathrm{d}s - \int_t^T Z_s \cdot \mathrm{d}X_s^c - \int_t^T \int_{\R^d\setminus\{0\}} H_t(x)\mu_X(\textrm{d}s,\textrm{d}x) , \quad t\in[0,T].
\end{equation}

\begin{definition}
A solution to the BSDE \eqref{eq:bsde-single-agent} is a triple $(Y,Z,H)$ such that $Y$ is a continuous $\F-$semimartingale satisfying \eqref{eq:bsde-single-agent}, and $(Z,H)\in\V^1(X):=\H^{2,d}_\text{loc}(X)\times \L^{1}_{\text{loc}}(\mu_X)$, where
\begin{align*}
& \H^{2,d}_\text{loc}(X) = \bigg\{  \R^d-\text{valued, }\F-\text{predictable processes } Z \text{ such that } \int_0^T |\sigma_s^\top(X_s)Z_s|^2 ~ \mathrm{d}s <\infty, ~ \P-\text{a.s.}   \bigg\}, \\
& \L^{1}_{\text{loc}}(\mu_X)  = \bigg\{    \R\text{-valued, }\F\text{-predictable function } H:  \sum_{s\leq \cdot}   \bigg(\int_{\R^d\setminus\{0\}} H_s(x)\mu_{X}(\mathrm{d}t,\mathrm{d}x) \bigg)^\frac{1}{2} \text{ is }\P\text{-loc. integrable} \bigg\}.
\end{align*}
\end{definition}

Our second and last assumption in this section will be that the map $F$ provides well-possedness of the forward version of BSDE \eqref{eq:bsde-single-agent}, which will be used to obtain a convenient representation of the set of admissible contracts.

\begin{assumption}\label{ass:nice-F-single-agent}
The map $F$ is such that for every $y\in\R$ and $(Z,H)\in\V^1(X)$ there exists a unique strong solution to the following SDE
\[
Y^{y,Z,H,\chi}_t = y - \int_0^t F_s(X_s,Y^{y,Z,H,\chi}_s,Z_s,H_s,\chi_s) \mathrm{d}s + \int_0^t Z_s \cdot \mathrm{d}X^c_s + \int_0^t \int_{\R^d\setminus\{0\}} H_t(x)\mu_X(\mathrm{d}s,\mathrm{d}x), \quad t\in[0,T]. 
\]
%$(ii)$ For every $y_1,y_2\in\R$ and $(Z,H)\in\V^1(X)$ there exists a constant $C$ such that for every $t\in[0,T]$
%\[
%|Y^{y_1,Z,H,\chi}_t - Y^{y_2,Z,H,\chi}_t | \leq C |y_1-y_2|.
%\]
\end{assumption}

\begin{remark}\label{rem-rho-bounded}
A sufficient condition for Assumption \ref{ass:nice-F-single-agent} is that the map $F$ is Lipschitz. This holds true for instance if $\rho$ is bounded.
\end{remark}

%
%In order to present the main result of this section, we introduce a family of processes corresponding to the dynamics of BSDE  \eqref{eq:bsde-single-agent}. For $y\in\R$ and $Z\in\H^{2,d}_\text{loc}(X)$, we define the process $Y^{y,Z,\chi}$ as the solution to the following SDE\footnote{Note that the process is well defined because $F$ is Lipschitz in the $Y$ variable, since the map $\rho$ is bounded.}
%\[
%Y^{y,Z,\chi}_t = y + \int_0^t F_s(X_s,Y^{y,Z,\chi}_s,Z_s,\chi_s) \textrm{d}s + \int_0^t Z_s \cdot \textrm{d}X_s, \quad t\in[0,T]. 
%\]
We define now the class of processes to which we will restrict our attention. 
\begin{definition} (i) For $y\in\R^N$ and $(Z,H)\in \V^1(X)$, we define the process $a^{\star,y,Z,H,\chi}$ by
\[
a_s^{\star,y,Z,H,\chi} := a^{\star}(s,X_s,Y^{y,Z,H,\chi}_s,Z_s,H_s,\chi_s), \quad \emph{d}t\otimes \emph{d}\P-\text{a.s. over } [0,T]\times \Omega.
\]
(ii) For $y\in\R^N$, we denote by $\Vc^{y,\chi}$ the following class of processes
%\begin{align*}
%\Zc^{y,\chi}:= \bigg\{   Z\in\H^{2,N\times dN}_\text{loc}(X):  \text{ the process } a^{\star,y,Z,\chi} \in\Ac; ~ \forall i\in I, Y_T^{y,Z,\chi,i} \in \mathrm{Im}(\Uc_A^i), \P-\mathrm{a.s.}  \text{ and }  \\
%\forall i\in I, \forall \alpha^i\in\Ac^i( a^{\star,y,Z,\chi,-i}), ~ \exists q_{\alpha^i} >1 \text{ such that } ~~~~&  \\ 
% \frac{1}{q_{\alpha^i}} + \frac{1}{\hat p_{\alpha^i \otimes_i a^{\star,Z,\chi,-i}}} < 1, ~~  \sup_{\tau\in\Tc_{0,T}} \E^{\P^{\alpha^i\otimes_i a^{\star,y,Z,\chi,-i}}}[| Y_\tau^{y,Z,\chi,i} |^{q_{\alpha^i}}] < \infty \bigg\}. & 
%\end{align*}

\begin{align*}
\Vc^{y,\chi}:= \bigg\{   (Z,H)\in\V^1(X):  a^{\star,y,Z,H,\chi} \in\Ac; ~ Y_T^{y,Z,H,\chi} \in \mathrm{Im}(\Uc_A) ~ \P-\mathrm{a.s.}  \text{ and }   \forall \alpha\in\Ac ~ \exists q_{\alpha} >1 ~~~~&  \\ 
~~~ \text{ such that } \frac{1}{q_{\alpha}} + \frac{1}{\hat p_{\alpha}} < 1, ~ \sup_{\tau\in\Tc_{0,T}} \E^{\P^{\alpha}}[| Y_\tau^{y,Z,H,\chi} |^{q_{\alpha}}] < \infty \bigg\}. & 
\end{align*}
\end{definition}

\begin{remark}
If the function $b$ is bounded and $\lambda$ is bounded by an $F-$integrable function, then the integrability condition in the class $\Vc^\chi$ reduces to the existence of some $p>1$ such that $\sup_{\tau\in\Tc_{0,T}} \E [|Y_\tau^{0,Z,\chi}|^{p}] < \infty$. This is a consequence of H\"older's inequality and the densities between the probability measures having finite moments of any order.
\end{remark}

\medskip
We present now the main result of this section. As usual in the literature, we establish an equivalence between solving the problem of the agent and solving BSDE \eqref{eq:bsde-single-agent}. %which represents the value of the agent as the solution to this BSDE.

\begin{proposition}\label{prop-single-agent-bsde}
For every $\alpha^\star\in\Ac^\star(\xi,\chi)$ there exists a solution $(Y,Z,H)$ to BSDE \eqref{eq:bsde-single-agent}, with $(Z,H)\in\Vc^{Y_0,\chi}$, such that 
\begin{equation}\label{eq:a-star-single-agent-bsde}
\alpha^\star_s = a^\star(s,X_s,Y_s,Z_s,H_s,\chi_s), \quad \emph{d}t\otimes\emph{d}\P-\text{a.s. over } [0,T]\times \Omega.
\end{equation}
Conversely, let $(Y,Z,H)$ be a solution to BSDE \eqref{eq:bsde-single-agent}, with $(Z,H)\in\Vc^{Y_0,\chi}$. Then the control defined by \eqref{eq:a-star-single-agent-bsde} belongs to $\Ac^\star(\xi,\chi)$.
\end{proposition}

The previous proposition will allow us to reformulate the problem of the principal as a standard stochastic control problem, by restricting without loss of generality the form of the admissible contracts. The optimal response of the agent to the contracts in the class $\Vc^{\chi,y}$ follows immediately from Proposition \ref{prop-single-agent-bsde}.

\begin{corollary}\label{cor-optimal-response-single-agent-bsde}
Let $y\in\R$, $(Z,H)\in\Vc^{\chi,y}$ and $\xi= \Uc_A^{-1}( Y_T^{y,Z,H,\chi})$. Then the value of the agent is given by $V_A(\xi,\chi)=y$ and his optimal effort is given by $a^{\star,y,Z,H,\chi}$.
\end{corollary}

By combining all the results in this section, we obtain the reformulation of the problem of the principal. 
\begin{equation*} \label{eq:reformulated-principal-single-agent-bsde}
V_P = \sup_{y \geq R_0, \chi\in\Xi_2} \sup_{(Z,H)\in\Vc^{y,\chi}} ~  \E^{\P^{a^{\star,y,Z,H,\chi}}} \left[ e^{-\int_0^T r_s(X_s) \mathrm{d}s } ~ \Uc_P \bigg(L(X_T) - \Uc_A^{-1}(Y_T^{y,Z,H,\chi}) \bigg) - \int_0^T e^{-\int_0^s r_s(X_u) \mathrm{d}u}  u_P(\chi_s) \mathrm{d}s \right].
\end{equation*}

\subsubsection{The case of exponential utility} \label{sec:exp-utility-single-agent-bsde}

Assume now that $u_A, c \equiv 0$ and $\Uc_A(x)=-\exp(-R_A x)$, with $R_A > 0$. In the case of exponential utility the map $\rho$ usually takes the form $\rho_s(x,k,a)=R_A(\hat{u}_A(s,x,k)-\hat{c}_s(x,a))$, for some utility and cost functions $\hat u_A$ and $\hat c$. However, we do not impose $\rho$ to have this form. 
%and we just assume it is function bounded by below.
%\begin{assumption}\label{ass-rho-cert-equiv}
%The map $\rho:\Omega\times [0,T] \times \R^d \times\R\times A \longrightarrow\R$ is bounded by below.
%\end{assumption}

\medskip
In this setting we can actually obtain an alternative representation of the admissible contracts, by using the certainty equivalent process of the agent. To this end, we introduce the functions $G:\Omega\times[0,T]\times\R^d\times\R^d\times\Bc(\R^d;\R)\times\R \longrightarrow\R$ and $g:\Omega\times[0,T]\times\R^d\times\R^d\times\Bc(\R^d;\R)\times\R \times A \longrightarrow\R$ defined by 
\begin{align*}
g_s(x,z,h,k,a) &=   	\frac{1}{R_A}\rho_s(x,k,a) + z^\top\sigma_s(x)b_s(x,a) - \frac{1}{2}R_A ||\sigma_s(x)^\top z||^2 + \frac{1}{R_A}   \int_{\R^d\setminus\{0\}}\big(1-e^{R_A h(u)})\eta_t^{a}(x,\textrm{d}u), \\
G_s(x,z,h,k)  &=  \sup_{a\in A} g_s(x,z,h,k,a). 
\end{align*}
Similar to the previous section, we assume that the function $g$ is maximized at a unique point.
\begin{assumption}\label{ass-unique-max-g-single-agent}
For every $(s,x,z,h,k)\in [0,T]\times\R^d\times\R^d\times\Bc(\R^d;\R)\times\R$ the map $a\mapsto g_s(x,z,h,k,a)$ has $\P-$a.s. a unique maximizer over the set $A$, denoted by $\hat a^\star(s,x,z,h,k)$.
\end{assumption}

Consider next the following BSDE
\begin{equation}\label{eq:bsde-single-agent-cert-equiv}
Y_t = \xi + \int_t^T G_s(X_s,Z_s,H_s,\chi_s) \mathrm{d}s - \int_t^T Z_s \cdot \mathrm{d}X^c_s + \int_t^T \int_{\R^d\setminus\{0\}} H_s(x)\mu_X(\textrm{d}t,\textrm{d}x) , \quad t\in[0,T].
\end{equation}
For $y\in\R$ and $Z\in\H^{2,d}_\text{loc}(X)$, we define the process $Y^{y,Z,\chi}$ as follows\footnote{Note this is a direct definition since $Y$ does not appear in the right-hand side.}
\[
Y^{y,Z,H,\chi}_t := y - \int_0^t G_s(X_s,Z_s,H_s,\chi_s) \textrm{d}s + \int_0^t Z_s \cdot \mathrm{d}X^c_s - \int_0^t \int_{\R^d\setminus\{0\}} H_s(x)\mu_X(\textrm{d}t,\textrm{d}x) , \quad t\in[0,T]. 
\]
We define now the class of processes to which we will restrict our attention. In the previous section we reformulated the problem of the principal through classes of processes associated to a fixed value of the agent. In the setting of exponential utility, this reformulation can be achieved through a single class of processes.

\begin{definition} We denote by $\hat \Vc^\chi$ the class of processes $(Z,H)\in\V^1(X)$ satisfying

\medskip
$(i)$ $\forall \alpha\in\Ac, ~\exists q_\alpha >1$ such that $\frac{1}{q_\alpha} + \frac{1}{\hat p_\alpha} < 1$ and $\sup_{\tau\in\Tc_{0,T}} \E^{\P^\alpha}[|\Uc_A(Y_\tau^{0,Z,H,\chi})|^{q_\alpha}] < \infty$.

$(ii)$ The process $a^{\star,Z,H,\chi} \in\Ac$, where
\[
a^{\star,Z,H,\chi}_s := \hat a^\star(s,X_s,Z_s,H_s,\chi_s), \quad \emph{d}t\otimes\emph{d}\P-\text{a.s. over } [0,T]\times \Omega.
\]
\end{definition}

We have then an alternative representation of the set $\Ac^\star(\xi,\chi)$, which will be used in this setting. The following is an analogous to Proposition \ref{prop-single-agent-bsde}. %which represents the value of the agent as the solution to this BSDE.

\begin{proposition}\label{prop-single-agent-bsde-cert-equiv}
For every $\alpha^\star\in \Ac^\star(\xi,\chi)$ there exists a solution $(Y,Z,H)$ to BSDE \eqref{eq:bsde-single-agent-cert-equiv}, with $(Z,H)\in\hat\Vc^\chi$, such that
\begin{equation}\label{eq:a-star-single-agent-bsde-cert-equiv}
\alpha^\star_s = a^{\star,Z,H,\chi}_s , \quad \emph{d}t\otimes\emph{d}\P-\text{a.s. over } [0,T]\times \Omega.
\end{equation}
Conversely, let $(Y,Z,H)$ be a solution to BSDE \eqref{eq:bsde-single-agent-cert-equiv}, with $(Z,H)\in\hat\Vc^\chi$. Then the control $a^{\star,Z,H,\chi}$ belongs to $\Ac^\star(\xi,\chi)$.
\end{proposition}

We present next the analogous result to Corollary \ref{cor-optimal-response-single-agent-bsde}.

\begin{corollary}\label{cor-optimal-response-single-agent-bsde-cert-equiv}
Let $y\in\R$, $(Z,H)\in\hat\Vc^\chi$ and $\xi=Y_T^{y,Z,H,\chi}$. Then the value of the agent is given by $V_A(\xi,\chi)=\Uc_A(y)$ and his optimal effort is given by $a^{\star,Z,H,\chi}$.
\end{corollary}

To conclude this section, we obtain the reformulation of the problem of the principal. 
\begin{equation} \label{eq:reformulated-principal-single-agent-bsde-cert-equiv}
V_P = \sup_{ y \geq \Uc_A^{-1}(R_0) } \sup_{\chi\in\hat\Xi_2} \sup_{(Z,H)\in\hat\Vc^{\chi}} ~  \E^{\P^{a^{\star,Z,H,\chi}}} \left[ e^{-\int_0^T r_s(X_s) \mathrm{d}s } ~ \Uc_P \bigg(L(X_T) - Y_T^{y,Z,H,\chi} \bigg) - \int_0^T e^{-\int_0^s r_s(X_u) \mathrm{d}u}  u_P(\chi_s) \mathrm{d}s \right].
\end{equation}

\section{Applications in the literature} \label{sec:examples}

The purpose of this section is to show that our set of assumptions allows to treat directly different papers in the literature and to recover their results. The main point is that our weaker assumptions are satisfied in those works and therefore we can apply our main results to characterize the equilibrium (solution) to the agents' game (problem) and to reformulate the problem of the principal.

\subsection{H\"olmstrom and Milgrom (1987)}

In the continuous model by \citeauthor*{holmstrom1987aggregation} \cite{holmstrom1987aggregation}, the $d-$dimensional output controlled by the agent takes the form
\[
\textrm{d}X_t = \alpha_t \textrm{d}t + \Sigma^{\frac{1}{2}}\textrm{d}W^\alpha_t,
\]
where $W$ is an $n-$dimensional standard Brownian motion. The agent's cost function $c:A\rightarrow\R$ is assumed to be convex, continuously differentiable and with bounded derivative. $A\subset\R^d$ is assumed to be open and compact. The principal offers the agent the terminal remuneration $\xi$, which is an $\F_T^X$-adapted random variable, and the problem of the agent is given by
\[
V_0(\xi) = \sup_{\alpha\in\Ac} \E^{\P^\alpha}\bigg[ u\bigg(\xi-\int_0^T c(\alpha_t)\textrm{d}t\bigg)\bigg],
\]
with the utility function $u(y)=-\exp(-ry)$. 

\medskip
Our formulation recovers the problem of the agent with exponential utility by choosing $u_A=c=0$ and $\rho(a)=-rc(a)$. Since the effort is assumed to be bounded, the conditions for a contract to be admissible (Definition \ref{def-admisible-contracts-single-agent}) can be checked only under the initial measure $\P$. Condition $(iii)$ is automatically satisfied and $(ii)$ holds since $c$ is continuous and the effort is bounded. Similarly, condition $(i)$ reduces to the contract being sufficiently integrable and $(iv)$ is the usual incentive compatibility condition.

\medskip
The result \cite[Theorem 6]{holmstrom1987aggregation} says that if an agent is offered the contract $\xi$ and certainty equivalent $y$, then a necessary condition for the optimality of the control $\alpha^\star$ is that
\[
\xi = y + \int_0^T \bigg( c(\alpha^\star) - c'(\alpha^\star_t)^\top \alpha^\star_t + \frac{r}{2} c'(\alpha^\star_t)^\top \Sigma c'(\alpha^\star_t) \bigg) \textrm{d}t + \int_0^T c'(\alpha^\star_t)^\top  \textrm{d}X_T,
\]
which we can deduce from Proposition \ref{prop-single-agent-bsde-cert-equiv}, given that the optimality condition \eqref{eq:a-star-single-agent-bsde-cert-equiv} characterizes the optimal control through $Z_t = c'(\alpha^\star_t)$. Our result extends \cite[Theorem 6]{holmstrom1987aggregation} by establishing also the sufficiency of the previous representation,  that \cite{holmstrom1987aggregation} conjectures but only states for constant controls.

\subsection{Capponi and Frei (2015)}

In the discontinuous model by \citeauthor*{capponi2015dynamic} \cite{capponi2015dynamic}, the agent controls a diffusion with jumps which suffers the shocks from a compound Poisson process $J_t = \sum_{i=1}^{N_t} Y_i$, where $(Y_i)$ is a sequence of bounded non-negative i.i.d. random variables with cumulative distribution function $F$. The agent controls $\alpha=(u,\lambda)$, where $\lambda_t$ is the instantaneous intensity of $N_t$ and the output has the dynamics
\[
\textrm{d}X_t = (a_t + u_t) \textrm{d}t + \sigma_t \textrm{d}W_t^\alpha - \int_{\R_+} e \ \mu_J(\textrm{d}e,\textrm{d}t).
\]
The problem of the agent, when offered a terminal remuneration $\xi$ at a random time $\tau$ by the principal, is given by
\[
V_0(\xi) = \sup_{\alpha\in\Ac} \E^{\P^\alpha}\bigg[ u\bigg(\xi-\int_0^\tau c(\alpha_t)\textrm{d}t\bigg)\bigg],
\]
with the utility function $u(y)=-\exp(-ry)$ and the cost function $c$ twice continuously differentiable and strictly convex. The control $u$ is admissible if it is $(\Fc_t)-$predictable and bounded, whereas $\lambda$ is asked to be $(\Fc_t)-$predictable and bounded away from zero and infinity. Notice that under $\P^\alpha$ the compensator of $\mu_J(\textrm{d}e,\textrm{d}t)$ is given by $\lambda_t F(\textrm{d}e)\textrm{d}t$.

\medskip
The problem with exponential utility is covered by our model by choosing $u_A=c=0$ and $\rho(a)=-rc(a)$. Since the efforts are bounded, the admissibility conditions for the contracts in Definition \ref{def-admisible-contracts-single-agent} can be checked only under the initial probability $\P$. Conditions $(iv)$ and $(iii)$ are directly satisfied, while $(i)$ and $(iv)$ reduce to the integrability of the contract and the incentive compatibility conditions.

\medskip
The result \cite[Lemma 3.1]{capponi2015dynamic} says that the optimal contract, which gives certainty $y$ to the agent, takes the form
\[
\xi^\star = y + \int_0^\tau Z_t \textrm{d}X_t^c + \int_0^\tau \bigg( \frac{r}{2} Z_t^2\sigma_t^2 + \tilde H_t \lambda^\star_t m + c(\alpha^\star_t) - Z_t(a_t+u^\star_t) \bigg) \textrm{d}t  - \frac{1}{r} \sum_{0\leq t <\tau} \log\big( 1+r \tilde H_t \Delta J_t \big), 
\]
% their x_t = a_t + \sigma_t dW_t, so if I move a_t to the LHS and rename it \hat X_t = x_t - a_t then I get (\hat X_t)^c = \sigma_t dW_t
with $m=\int_{\R_+} y \ \textrm{d}F(y)$ and for some predictable processes $Z$ and $\tilde H$. Notice that this characterization is the same as the one provided by Proposition \ref{prop-single-agent-bsde-cert-equiv}, if we make the reverse change of variables $H_t(x):=\frac{1}{r} \log\big( 1+\tilde H_t(x)\big) $ and we assume that $\tilde H(x)=\tilde H_t x$ is linear. Indeed, in such case we have the equality (for the last term of the generator $g$)
\[
\int_{\R_+} \tilde H_t (x) \nu_t^{\alpha_t}(\textrm{d}x) = \int_{\R_+} \tilde H_t x \ \lambda_t F(\textrm{d}x)\textrm{d}t = \tilde H_t \lambda_t m. 
\]
However, we believe there is no reason to assume that the function $H_t$ (coming from the martingale representation) is linear for a general compound Poisson process and such assumption would impose a restriction over the set of admissible contracts. Therefore our result corrects the one presented in \cite[Lemma 3.1]{capponi2015dynamic}.

\medskip
Another comment is that for the reformulation of the problem of the principal, it is assumed in \cite{capponi2015dynamic} that $Z$ is bounded and $H$ is non-negative and bounded. As the authors say, this imposes a further restriction on the class of admissible contracts since those conditions are stronger than the ones defining the class $\hat \Vc$.

\subsection{Elie and Possama\"i (2019)}

\citeauthor*{elie2019contracting} \cite{elie2019contracting} study the continuous multi-agent problem with $N$ competitive agents. In their work, the output $X$ has the dynamics
\[
\textrm{d}X_t = b_t(X_t,\alpha_t) \textrm{d}t + \Sigma_t \textrm{d}W_t^\alpha.
\]
The best response of the agent $i$ to the action $\alpha^{-i}$ of the others, when offered a terminal remuneration $\xi^i$ by the principal, is given by
\[
V_0^i(\alpha^{-i},\xi^i) = \sup_{\alpha^i\in\Ac^i(\alpha^{-i})} \E^{\P^\alpha}\bigg[ u_i\bigg(\xi^i+\Gamma_i(X_T)-\int_0^T c_t(X_t,\alpha^i_t)\textrm{d}t\bigg)\bigg],
\]
with the utility function $u_i(y)=-\exp(-r_i y)$ and the map $\Gamma_i$ representing a comparison made by the agent $i$ of his performance with the others. Notice that the game between the agents $i$ falls into our setting by considering $\hat\xi^i = \xi^i+\Gamma_i(X_T)$. 

\medskip
The assumptions in \cite{elie2019contracting} are pretty strong and our model relaxes some of them and removes the unnecessary ones. The admissible efforts are assumed to be such that the density measures have finite moments of any order, in which case the admissibility conditions for the contracts can be checked only under the initial probability $\P$. They are also assumed to be such that $\int_0^T c_t(X_t,\alpha^i_t)\textrm{d}t$ belongs to the Morse-Transue space $M^\phi(\R^N)$, which implies condition $(ii)$ in Definition \ref{def-admisible-contracts}. As the problem is of exponential utility, condition $(iii)$ is trivially satisfied. The admissible contracts $\xi$ are assumed to belong to $M^\phi(\R^N)$ which is more restrictive than condition $(i)$.  

\medskip
For such admissible contracts, \cite[Theorem 4.1]{elie2019contracting} establishes a one-to-one correspondence between sufficiently integrable Nash equilibria $\alpha^\star\in\Ac^\star(\xi)$ and solutions $(Y,Z)$ to BSDE \eqref{eq:bsde-agent-cert-equiv} such that $Z\in\H^2_{BMO}(\Mc^{N,N})$. In the case $N=1$, such condition on $Z$ implies the boundedness of $Y$ and therefore it is a stronger condition than $Z\in\hat\Vc$, (see for instance \cite[Proposition 2.2]{briand2013simple}). In the multidimensional case we conjecture the same to be true, so their characterization of Nash equilibria holds over a smaller class.

\subsection{El Euch, Mastrolia, Rosenbaum and Touzi (2020)}

This paper does not fit exactly into our setting due to the agent having a different objective function. However, we show how to adapt our method and to derive the corresponding BSDE associated to the problem of the agent. 

\medskip
\citeauthor*{euch2021optimal} \cite{euch2021optimal} consider a single agent controlling a process $X=(X^S,X^a,X^b)\in\R^3$. Here, $X_t^S=X_0^S + \sigma W_t$ is an uncontrolled continuous diffusion, $X^a_t=N_t^a$ and $X^b_t=N_t^b$ are independent Poisson processes whose intensities are controlled. Namely, the agent chooses an effort $\alpha=(\alpha^a,\alpha^b)$, uniformly bounded, which results in intensities $\lambda(\delta_t^a)$ and $\lambda(\delta_t^b)$ respectively for $N_t^a$ and $N_t^b$. The problem of the agent is given by
\[
V_0(\xi) = \sup_{\alpha\in\Ac} \E^{\P^\alpha}\bigg[ u\bigg(\xi+\int_0^T \big( \alpha_t^a \textrm{d}N_t^a + \alpha_t^b \textrm{d}N_t^b + (N_t^b-N_t^a) \textrm{d}S_t \big)  \bigg)\bigg],
\]
with the utility function $u(y)=-\exp(-R_A y)$. This problem falls into the exponential setting of Section \ref{sec:exp-utility-single-agent-bsde} if we make the replacement in \eqref{eq:agent-problem}
\[
\int_0^T \rho_s(X_s,\alpha_s)\textrm{d}s \mapsto \text{ replace by }  R_A \bigg(\int_0^T \rho^S(X_s)\textrm{d}X_s^c + \int_0^T\int_\R \rho^a(\alpha_s) \mu_{X^a}(\textrm{d}s,\textrm{d}u) + \int_0^T \int_\R \rho^b(\alpha_s) \mu_{X^b}(\textrm{d}s,\textrm{d}u) \bigg),
\]
with the functions 
\[
\rho^S(x)=(x^b-x^a,0,0)^\top, ~ \rho^a(\alpha)=\alpha^a, ~ \rho^b(\alpha)=\alpha^b.
\]
Of course, to adapt our result to this new setting we would need to update the conditions for the admissible contracts. In this case condition $(ii)$ in Definition \ref{def-admisible-contracts-single-agent} becomes 
\[
\sup_{\tau\in\Tc_{0,T}} \E^\P\bigg[ e^{-\hat p_\alpha R_A \big(\int_0^\tau \rho^S(X_s)\textrm{d}X_s^c + \int_0^\tau \int_\R \rho^a(\alpha_s) \mu_{X^a}(\textrm{d}s,\textrm{d}u) + \int_0^\tau \int_\R \rho^b(\alpha_s) \mu_{X^b}(\textrm{d}s,\textrm{d}u) \big) } \bigg] < \infty,
\]
which holds true because $N_t^b-N_t^a$ is assumed to be bounded, as well as the intensities of the Poisson processes. Similarly, the integrability condition needs to be checked only under the initial probability $\P$ because the probability densities have finite moments of any order. Condition $(iii)$ is satisfied trivially in the setting of exponential utility and condition $(i)$ matches the integrability of the contract required in \cite{euch2021optimal}. 

\medskip
The martingale representation in this setting takes the form
\[
Z_t \textrm{d}X_t^c + \int_\R H_s^a(u) \mu_{X^a}(\textrm{d}s,\textrm{d}u) + \int_\R H_s^b(u) \mu_{X^b}(\textrm{d}s,\textrm{d}u) = Z_t^S \textrm{d}S_t+\tilde Z_t^a\textrm{d}N_t^a+\tilde Z_t^b\textrm{d}N_t^b,
\]
by defining $\tilde Z_t^a = H_t^a(1)$, $\tilde Z_t^b = H_t^b(1)$. Combining this with the form of $\rho$ and by putting the corresponding terms together we obtain the generator
\begin{align*}
g(x,z,h,\alpha)  & =   	- \frac{1}{2}R_A ||\sigma^\top (z+\rho^S(x))||^2 + \frac{1}{R_A}  \int_{\R\setminus\{0\}}\big(1-e^{R_A ( h^a(u)-\rho^a(\alpha))} \big) \lambda(\delta_s^a) \eta^{a}(\textrm{d}u)  \\
& ~ ~~ +  \frac{1}{R_A} \int_{\R\setminus\{0\}}\big(1-e^{R_A ( h^b(u)-\rho^b(\alpha))}\big)\lambda(\delta_s^b)\eta^{b}(\textrm{d}u), 
\end{align*}
with $\eta^a$ and $\eta^b$ Dirac masses at $\{1\}$. Then, the BSDE
\[
Y_t = \xi + \int_t^T G(X_s,Z_s,H_s) \mathrm{d}s - \int_t^T Z_s \cdot \mathrm{d}X^c_s + \int_t^T \int_\R H_s^a(u) \mu_{X^a}(\textrm{d}s,\textrm{d}u) + \int_t^T \int_\R H_s^b(u) \mu_{X^b}(\textrm{d}s,\textrm{d}u) , \quad t\in[0,T],
\]
leads to the reformulation of the admissible contracts stated in \cite[Theorem 3.1(i)]{euch2021optimal}.

%\subsection{Baldaci, Possama\"i and Rosenbaum}

%\subsection{Kharroubi, Lim and Mastrolia}

\section{Smoothness  of the value function of the principal} \label{sec:principal}

Given the results of the previous section, we can reformulate the problem of the principal as a standard stochastic control problem.  As showed by \citeauthor*{cvitanic2018dynamic} \cite{cvitanic2018dynamic}, the value function of the principal can thus be identified as a viscosity solution to the corresponding Hamilton-Jacobi-Bellman equation. In this section, we extend the approach to the case with jumps for the particular setting of exponential utility of the agents, and we use it to study the smoothness of the value function of the principal. We present all the proofs in Appendix \ref{app:proofs-principal}.

\medskip
Throughout all this section we will assume thus the following conditions, that allow to obtain a well defined HJB equation.

\begin{assumption}\label{ass:deterministic}
$(i)$ All the functions considered in the model are deterministic, that is, they do not depend on $\omega$. \\
$(ii)$ For every $(t,x)\in[0,T]\times\R^{dN}$ and $i\in I$ we have $\int_{E\setminus\{0\}} |\beta^i_t(x,e)|F^i(\mathrm{d}e) <\infty$. 
\end{assumption}

\subsection{The case of exponential utility with risk-neutral principal}

We place ourselves in the same setting as in Section \ref{sec:exp-utility}, that is, we assume that each function $u_A^i, c^i \equiv 0$ and $\Uc_A^i(x)=-\exp(-R_A^i x)$, with $R_A^i > 0$. We also consider a risk-neutral principal, that is we take $r\equiv 0$ and $\Uc_p(x)=x$.

\medskip
In the setting just described, the reformulated problem of the principal \eqref{eq:reformulated-principal-cert-equiv} becomes
\begin{equation*} 
V_P = \sup_{ y \geq \Uc_A^{-1}(R_0) } \sup_{\chi\in\hat\Xi_2} \sup_{(Z,H)\in\hat\Vc^{\chi}} ~  \E^{\P^{a^{\star,Z,\chi}}} \left[  L(X_T) - \sum_{i=1}^N Y_T^{y,Z,H,\chi,i}  - \int_0^T  u_P(\chi_s) \mathrm{d}s \right],
\end{equation*}
with the dynamics of the controlled processes 
\begin{align*}
 X_t &= X_0 +  \int_0^t \Sigma_s(X_s) b_s(X_s,a^{\star,Z,\chi}_s) \mathrm{d}s + \int_0^t \Sigma_s(X_s) \mathrm{d}W_s^{a^{\star,Z,\chi}} +  \int_0^t\int_{E\setminus\{0\}} \beta_s(X_s,e) \vec{\mu}_{J}(\textrm{d}s,\textrm{d}e) , \quad t\in[0,T], \\
 Y^{y,Z,H,\chi,i}_t & = \int_0^t \frac{1}{R_A^i} \bigg( \frac{(R_A^i)^2}{2} ||Z_s^{i,:}\Sigma_s(X_s)||^2 - \rho^i_s(X_s,\chi_s,a^{\star,Z,H,\chi}_s) -  \sum_{\ell=1}^N  \int_{\R^d\setminus\{0\}}\big(1-e^{R_A^i H_s^{i,\ell}(u)})\eta_s^{\ell,a^{\star,Z,H,\chi}}(X_s,\textrm{d}u)  \bigg) \textrm{d}s \\
& ~~ + y^i +  \int_0^t Z_s^{i,:} \Sigma_s(X_s) \mathrm{d}W_s^{a^{\star,Z,H,\chi}}  - \int_0^t \int_{\R^d\setminus\{0\}} H^i_s(u)\vec{\mu}_X(\textrm{d}s,\textrm{d}u), ~ t\in[0,T],~ i\in I. 
\end{align*}
% Y^{y,Z,H,\chi}_t := y - \int_0^t G_s(X_s,Z_s,H_s,\chi_s) \textrm{d}s + \int_0^t Z_s \mathrm{d}X^c_s - \int_0^t \int_{\R^d\setminus\{0\}} H_s(x)\vec{\mu}_X(\textrm{d}t,\textrm{d}x) , \quad t\in[0,T]. 

Since the dependence on the $y$-variable is linear, the previous problem is equivalent to 
\begin{align} \label{eq:reformulated-principal-cert-equiv-2}
V_P  = \sup_{\chi\in\hat\Xi_2} \sup_{(Z,H)\in\hat\Vc^{\chi}} ~  \E^{\P^{a^{\star,Z,H,\chi}}}  \bigg[  L(X_T) & + \sum_{i=1}^N \int_0^T  \bigg( \frac{1}{R_A^i}\rho^i_s(X_s,\chi_s,a^{\star,Z,H,\chi}_s)  - \frac{1}{2}R_A^i ||Z_s^{i,:}\Sigma_s(X_s)||^2- \frac{1}{N} u_P(\chi_s)  \bigg)\mathrm{d}s \\
&  \nonumber + \sum_{i=1}^N \sum_{\ell=1}^N  \int_0^T\int_{\R^d\setminus\{0\}}\bigg( \frac{1-e^{R_A^i H_s^{i,\ell}(u)}}{R_A^i}  + H_s^{i,\ell}(u)\bigg)\eta_s^{\ell,a^{\star,Z,H,\chi}}(X_s,\textrm{d}u)\mathrm{d}s  \bigg] - \tilde R_0 , 
\end{align}
where $\tilde R_0:= \sum_{i=1}^N (\Uc_A^i)^{-1}(R_0^i)$. Moreover, the dynamics of $X$ can be written equivalently as
\begin{equation*}
X_t = X_0 + \int_0^t \Sigma_s(X_s)b_s(X_s,\alpha_s) \mathrm{d}s + \int_0^t \Sigma_s(X_s) \mathrm{d}W_s^\alpha +  \sum_{i=1}^N \int_0^t\int_{E\setminus\{0\}} I_{i}[ \beta^i_t(X_s,e)] \mu_{J_i}(\textrm{d}s,\textrm{d}e) , \quad t\in[0,T],
\end{equation*}
where $I_i[x]$ maps $x\in\R^d$ into the $i-$th block of $\R^{dN}$.

\medskip
The Hamilton-Jacobi-Bellman equation associated to Problem \eqref{eq:reformulated-principal-cert-equiv-2} is given by
\begin{align}\label{eq:HJB-equation}
%-v_t(t,x)  - H_t(x,Dv(t,x),D^2v(t,x)) -  \Hc v(t,x)  = 0, (t,x)\in[0,t)\times\R^{dN}, \\
-v_t(t,x)  - H_t(x,Dv(t,x)) -  \frac{1}{2}\text{Tr}\big(\Sigma_t(x)\Sigma_t^\top(x) D^2v(t,x)\big) -  \Hc v(t,x)  = 0, (t,x)\in[0,t)\times\R^{dN}, \\
\nonumber v(T,x)= L(x), ~ x\in\R^{dN},
\end{align}
with the Hamiltonian functions $h:[0,T]\times\R^{dN}\times\R^{dN}\times\Mc^{N,dN}\times\Bc(\R^d;\Mc^{N,N})\times\R^N\times A \rightarrow\R$ and $H:[0,T]\times\R^{dN}\times\R^{dN} \rightarrow\R$, and the integro-differential operator $\Hc$ defined by
\begin{align*}
%H_t(x,p,M) & := \sup_{(z,h,k)\in\Mc^{N,dN}\times\Bc(\R^d;\Mc^{N,N})\times\R^N_+ }   ~ h_t(x,p,z,h,k,\hat a^\star(t,x,z,h,k)) + \frac{1}{2}\text{Tr}\big(\Sigma_t(x)\Sigma_t^\top(x)M\big) , \\
H_t(x,p) & := \sup_{(z,h,k)\in\Mc^{N,dN}\times\Bc(\R^d;\Mc^{N,N})\times\R^N_+ }   ~ h_t(x,p,z,h,k,\hat a^\star(t,x,z,h,k)), \\
h_t(x,p,z,h,k,a) & : =  \sum_{i=1}^N \bigg( \frac{1}{R_A^i}\rho^i_t(x,k,a)  - \frac{1}{2}R_A^i ||z^{i,:}\Sigma_t(x)||^2 + \sum_{\ell=1}^N \int_{\R^d\setminus\{0\}}\bigg( \frac{1-e^{R_A^i h^{i,\ell}(u)}}{R_A^i}  + h^{i,\ell}(u)\bigg)\eta_t^{\ell,a}(x,\textrm{d}u)\bigg) \\
& \hspace{.5cm} -u_P(k) + p \cdot \Sigma_t(x)b_t(x,a), \\
%\Hc v(t,x) &: = \sum_{i=1}^N \sum_{i=1}^N \int_{\R^d\setminus\{0\}} \big( v(t,x+ {\color{red}e_i[u]}) - v(t,x) \big)\delta^{i,k}\eta_t^i(x,\textrm{d}u). \\
\Hc v(t,x) &: = \sum_{i=1}^N \int_{E\setminus\{0\}} \bigg( v\big(t,x+I_{i}[\beta_t^i(x,e)] \big) -  v(t,x) \bigg) F^i(\textrm{d}e).
\end{align*}

\begin{remark}
Note that, by the change of variable formula, since $\eta_s^{i}$ is the push-forward measure of $F^i$, we have the equivalence for every $(t,x)\in [0,T]\times\R^{dN}$
\[
\int_{\R^d\setminus\{0\}} \big(  v\big(t,x+I_{i}[u] \big) - v(t,x) \big) \eta_t^i(x,\emph{d}u) \emph{d}t = \int_{E\setminus\{0\}} \bigg(  v\big(t,x+I_{i}[ \beta^i_t(x,e)] \big) - v(t,x) \bigg)  F^i(\emph{d}e)\emph{d}t.
\]
\end{remark}

We start by proving that the value function of the problem of the principal is the unique (in an appropriate class of functions) viscosity solution to the I-PDE \eqref{eq:HJB-equation}. To this end, we assume the following conditions that are not needed in the previous sections.

\begin{assumption}\label{ass:pham-jumps}
$(i)$ The map $\Sigma$ is continuous in $(t,x)$ and Lipschitz in $x$. \\
$(ii)$ The map $b$ is continuous in $(t,x,a)$ and the map $\Sigma b$ is Lipschitz in $x$. \\
$(iii)$ The map $\hat a^\star$ is continuous in $(t,x,z,h,k)$. \\
$(iv)$  $E$ is finite dimensional. \\
$(v)$ The map $\beta$ is continuous in $(t,x)$. For each $i\in I$ there exists $\rho_i:E\rightarrow\R_+$, with $\sum_{i=1}^N\int \rho_i^2(e)F^i(\emph{d}e)<\infty$, such that
\[
|\beta^i(t,x,e)-\beta^i(t,y,e)| \leq \rho_i(e)|x-y|, \quad |\beta^i(t,x,e)|\leq \rho_i(e)(1+|x|), \quad \forall t\in[0,T], ~\forall x,y\in\R^{dN}, ~\forall e\in E.
\]
$(vi)$ The map $L$ is Lipschitz. \\
$(vii)$ The map $\phi:[0,T]\times\R^{dN}\times\Mc^{N,dN}\times  \Bc(\R^d;\Mc^{N,N})\times\R^N  \longrightarrow\R$ defined below is Lipschitz in $(t,x)$ 
\begin{align*}
\phi(t,x,z,h,k) :=  &\sum_{i=1}^N   \bigg( \frac{1}{R_A^i}\rho^i_t(x,k,a^\star(t,x,z,h,k)))  - \frac{1}{2}R_A^i ||z^{i,:}\Sigma_t(x)||^2 \bigg)-  u_P(k)  \\
&   + \sum_{i=1}^N \sum_{\ell=1}^N \int_{\R^d\setminus\{0\}}\bigg( \frac{1-e^{R_A^i h^{i,\ell}(u)}}{R_A^i}  + h^{i,\ell}(u)\bigg)\eta_t^{\ell,a^\star(t,x,z,h,k)}(x,\mathrm{d}u).
\end{align*}
\end{assumption}

\begin{remark}
A sufficient condition for $\Sigma b$ to be Lipschitz in $x$ is that both $\Sigma$ and $b$ are Lipschitz in $x$ and one of the functions is bounded. A sufficient condition for the map $\hat a^\star$ to be continuous is that the sets $A_i$ are compact, the maps $g^i$ are continuous and depend on the $a$-variable only through $a^i$ (see for instance \cite[Theorem 17.31]{aliprantis2006infinite}).
\end{remark}

Let us define the value function of the principal starting from any initial condition $(t,x)\in[0,T]\times\R^{dN}$
\begin{equation*}
V(t,x) : = \sup_{\chi\in\hat\Xi_2} \sup_{(Z,H)\in\hat\Vc^{\chi}} ~  \E^{\P^{a^{\star,Z,H,\chi}}} \bigg[  L(X^{t,x}_T) + \int_t^T  \phi(s,X_s^{t,x},Z_s,H_s,\chi_s )\mathrm{d}s \bigg],
\end{equation*}
where $X^{t,x}$ is the solution to 
\[
 X_s^{t,x} = x  +  \int_t^s \Sigma_r(X_r) b_r(X_r^{t,x},a^{\star,Z,H,\chi}_r) \mathrm{d}r + \int_t^s \Sigma_r(X^{t,x}_r) \mathrm{d}W_r^{a^{\star,Z,H,\chi}}+  \int_t^s\int_{E\setminus\{0\}} \beta_r(X_r^{t,x},e) \vec{\mu}_{J}(\textrm{d}r,\textrm{d}e), \; r\in[ t,T].
\]
Notice, by definition, the equality $V_P = V(0,X_0)$. Then we have, by using the results in \citeauthor*{pham1998optimal} \cite{pham1998optimal}, the standard result of the value function being a viscosity solution to the HJB equation associated to the stochastic control problem.

\begin{proposition} \label{prop:unique-viscosity}
Under Assumption \ref{ass:pham-jumps}, the function $V$ is the unique viscosity solution of I-PDE \eqref{eq:HJB-equation} which is continuous in $[0,T]\times\R^{dN}$ and uniformly continuous in $x$, uniformly in $t$.
\end{proposition}

We will present next an alternative characterization of the viscosity solution to the HJB equation \eqref{eq:HJB-equation}, as the solution to an associated FBSDE system. To achieve this, we follow the results in \citeauthor*{barles1997backward} \cite{barles1997backward}. As shown below, this characterization will allow to study the smoothness of the value function through the smoothness of the aforementioned system.

\medskip
Define the map $\psi:[0,T]\times\R^{dN}\times\R^{nM}\rightarrow\R$  as follows
\begin{align*}
\psi(t,x,\zeta) := \sup_{(z,h,k)\in\Mc^{N,dN}\times\Bc(\R^d;\Mc^{N,N})\times\R^N_+ } ~ \phi(t,x,z,h,k) + \zeta \cdot b_t(x,\hat a^\star(t,x,z,h,k)).
\end{align*}

We assume again some conditions that are not needed for the previous results in the paper.

\begin{assumption} \label{ass:bbp-jumps} 
$(i)$ The maps $\Sigma$ and $(\beta^i)_{k=1}^L$ do not depend on $t$. \\
$(ii)$ There exists $C>0$ such that each $\beta^i$ satisfies
\[
|\beta^i(x,e)-\beta^i(y,e)| \leq C|x-y|(1 \wedge |e|), \quad |\beta^i(x,e)|\leq C(1+|x|)(1\wedge |e|), \quad \forall  x,y\in\R^{dN}, ~\forall e\in E.
\] 
$(iii)$ The map $\psi$ is Lipschitz in $\zeta$, uniformly in $(t,x)$. % and there exists $\hat C,p>0$ such that
%\[
%|\psi(t,x,0)| \leq \hat C(1+|x|^p), \quad \forall  t\in[0,T], ~\forall x\in\R^{dN}.
%\]
\end{assumption}

\begin{remark}
A sufficient condition for Assumption \ref{ass:bbp-jumps} $(iii)$ is that the map $b$ is bounded.
\end{remark}

\begin{proposition}\label{prop:fsbde-viscosity}
Under Assumptions \ref{ass:pham-jumps} and \ref{ass:bbp-jumps}, the map $u:[0,T]\times\R^{dN}\rightarrow\R$ defined by $u(t,x)=\tilde Y_t^{t,x}$ is  a viscosity solution to the I-PDE \eqref{eq:HJB-equation}, where $(\tilde X^{t,x},\tilde Y^{t,x}, \tilde Z^{t,x}, \tilde H^{t,x})$ is the adapted solution to the FBSDE system
\[
\tilde X_s^{t,x} = x +\int_t^s \Sigma(\tilde X_r^{t,x}) \mathrm{d}W_r + \sum_{i=1}^N \int_t^s \int_{E\setminus\{0\}} I_{i}[ \beta^i(\tilde X_s^{t,x},e)] \mu_{J_i}(\mathrm{d}r,\mathrm{d}e),
\]
\[
\tilde Y_s^{t,x} = L(\tilde X_T^{t,x}) + \int_s^T \psi(r,\tilde X_r^{t,x},\tilde Z_r^{t,x}) \mathrm{d}r - \int_s^T \tilde Z_r^{t,x} \mathrm{d}W_r - \sum_{i=1}^N \int_s^T \int_{E\setminus\{0\}} \tilde H_r^{t,x}(e) \big(\mu_{J_i}(\mathrm{d}r,\mathrm{d}e) -F^i(\mathrm{d}e)\mathrm{d}r\big).
\]
Moreover, if the maps $L$ and $\psi(t,\cdot,\zeta)$ are uniformly continuous, uniformly on $\zeta$, and bounded, then $u$ is uniformly continuous and bounded.
\end{proposition}

To conclude this section, we use the FBSDE representation of the value function to prove its smoothness under some additional assumptions. To do so, we follow the results by \citeauthor*{fujii2018quadratic} \cite{fujii2018quadratic}.

\begin{assumption}\label{ass:fujii-jumps} ~\\
$(i)$ The maps $\Sigma$ and $(\beta^i)_{k=1}^L$ are continuously differentiable with respect to $x$, with bounded derivative. \\
%$(ii)$ There exists a positive constant $\hat C_1$ such that for every $k\in\{1,\dots,L\}$ we have
%\[
%|\beta^i(0,e)| \leq \hat C_1(1 \wedge |e|), \quad |\delta_x \beta^i(x,e)| \leq \hat C_1(1\wedge|e|), \quad \forall x\in\R^{dN}, \forall e\in E.
%\]
$(ii)$ For every $(x,\zeta)$ there exist a constant $\hat C_2>0$ and a non-negative measurable function $l:[0,T]\rightarrow\R_+$ such that 
\[
|\psi(t,x,\zeta) | \leq l_t + \frac{\hat C_2}{2} |\zeta|^2, \quad \emph{d}t-a.e., ~ t\in[0,T].
\]
$(iii)$ The map $|L(x)| + l_t$ is bounded uniformly in $(t,x)$. \\
$(iv)$ The maps $L$ and $\psi$ are continuously differentiable with respect to $x$. \\
$(v)$ There exists a positive constant $\hat C_3$ such that for every $(t,x)\in[0,T]\times\R^{dN}$
\[
|\delta_x L(x)|\leq \hat C_3, \quad |\delta_x \psi(t,x,0)| \leq \hat C_3.
\]
$(vi)$ There exists a positive constant $\hat C_4$ such that for every $(t,x)\in[0,T]\times\R^{dN}$ and $\zeta_1,\zeta_2\in\Mc^{N,dN}$
\[
|\delta_x \psi(t,x,\zeta_1) - \delta_x \psi(t,x,\zeta_2)| \leq  \hat C_4 (1+|\zeta_1|+|\zeta_2|)|\zeta_1 - \zeta_2|.
\] 
\end{assumption}

\begin{proposition} \label{prop-fujii-jumps}
Under Assumptions \ref{ass:pham-jumps}, \ref{ass:bbp-jumps} and \ref{ass:fujii-jumps}, the map $u$ defined in Proposition \ref{prop:fsbde-viscosity} is continuous and continuously differentiable with respect to $x$.
\end{proposition}

We have finally, the last assumption which allows us to present the main result of the section.

\begin{assumption} \label{ass:last-one}
$(i)$ The map $\psi(t,\cdot,\zeta)$ is uniformly continuous, uniformly in $\zeta$. \\
$(ii)$ There exists a correspondence $\Vc^\star:[0,T]\times\R^{dN}\times\R^{dN}\rightrightarrows \Mc^{N,dN}\times  \Bc(\R^d;\Mc^{N,N})\times\R^N$ of maximizers of $h$, that is
\[
\Vc^\star(t,x,p) = \bigg\{ (z,h,k)\in\Mc^{N,dN}\times\Bc(\R^d;\Mc^{N,N})\times\R^N : H_t(x,p,M) = h_t(x,p,z,h,k,\hat a^\star(t,x,z,h,k))   \bigg\},
\] 
which has nonempty values.
\end{assumption}

\begin{remark}\label{rem:maximizers}
If for each $(t,x,p)$ the optimization in the Hamiltonian $H$ can be reduced to a compact set $C(t,x,p)$ (for instance if $h_t$ is coercitive) and the correspondence $C$ is continuous, then by the Maximum Theorem (see for instance \cite[Theorem 17.31]{aliprantis2006infinite}) we have that $\Vc^\star$ has nonempty compact values and it is upper-hemicontinuous.
\end{remark}

\begin{theorem} \label{thr-smoothness-and-existence}
 Under Assumptions \ref{ass:pham-jumps}, \ref{ass:bbp-jumps}, \ref{ass:fujii-jumps} and \ref{ass:last-one} we have: \\
$(i)$ The value function of the principal $V$ is continuously differentiable with respect to $x$. \\
$(ii)$ Define the processes 
\[
Z_t^\star := z^\star(t,X_t,\partial_x V(t,X_t)), ~ H_t^\star := h^\star(t,X_t,\partial_x V(t,X_t)), ~ \chi_t^\star := \chi^\star(t,X_t,\partial_x V(t,X_t)),
\] 
where $(z^\star,h^\star,\chi^\star)(t,x,p)$ is any measurable selection of maximizers of $h$ as in Assumption \ref{ass:last-one}. If $\chi^\star\in\hat\Xi_2$ and $(Z^\star,H^\star)\in\hat\Vc^{\chi^\star}$, then the optimal control for problem \eqref{eq:reformulated-principal-cert-equiv-2} is given by the triplet $(\chi^\star,Z^\star,H^\star)$.  In such case, the optimal contract for each agent $i\in I$ is given by 
\begin{align*}
\xi^{\star,i} & =  \int_0^t  \frac{1}{R_A^i} \bigg( \frac{1}{2}(R_A^i)^2 ||Z_s^{^\star,i,:}\Sigma_s(X_s)||^2 - \rho^i_s(X_s,\chi_s^\star,a^{\star,Z^\star,H^\star,\chi^\star}_s) - \sum_{\ell=1}^N  \int_{\R^d\setminus\{0\}}\big(1-e^{R_A^i H_s^{\star,i,\ell}(x)} \big)  \nu_s^{\ell,a^{\star,Z^\star,H^\star,\chi^\star}}(\mathrm{d}x)  \bigg) \mathrm{d}s \\
& ~~ + (\Uc_A^i)^{-1}(R_0^i) + \int_0^t Z_s^{\star,i,:} \Sigma_s(X_s) \mathrm{d}W_s^{a^{\star,Z^\star,H^\star,\chi^\star}}  - \int_0^t \int_{\R^d\setminus\{0\}} H^{\star,i}_s(x)\vec{\mu}_X(\mathrm{d}s,\mathrm{d}x), ~ t\in[0,T]. 
\end{align*}
\end{theorem}

\begin{remark}
A sufficient condition for the control $(\chi^\star,Z^\star,H^\star)$ to be admissible is that it is bounded, which happens for instance if it is independent of $X_t$, $\partial_x V(t,X_t)$ is bounded and $\Vc^\star$ is upper-hemicontinuous (see Remark \ref{rem:maximizers}).
\end{remark}

{\small
 \bibliographystyle{plain}
\bibliography{bibliographyDylan}}

\appendix

\section{Auxiliary results}

\subsection{Dynamic programming principle}

For $\tau\in\Tc_{0,T}$, an agent $i\in I$, actions of the others $\alpha^{-i}$ and $\alpha \in \Ac^i(\alpha^{-i})$, we define 
\begin{equation*} 
L^i(\tau,\alpha^{-i},\alpha^i) = \E_\tau^{\P^\alpha} \left[ e^{-\int_\tau^T \rho^i_s(X_s,\chi_s,\alpha_s) \mathrm{d}s} ~ \Uc^i_A (\xi^i) + \int_\tau^T e^{-\int_\tau^s \rho^i_u(X_u,\chi_u,\alpha_u) \mathrm{d}u} \big( u_A^i(\chi_s) - c^i_s(X_s,\alpha_s) \big) \mathrm{d}s   \right].
\end{equation*}
We therefore have, by definition
\begin{equation}\label{continuationutility1}
V^i(\tau,\alpha^{-i},\xi^i,\chi^i) = \esssup_{\alpha^i\in \Ac^i(\alpha^{-i})} L^i(\tau,\alpha^{-i},\alpha^i).
\end{equation}

We have then the following Dynamic Programming Principle, whose proof is inspired by the one in \cite{euch2021optimal}.
\begin{proposition}\label{prop-DPP}
Let $\tau,\theta \in\Tc_{0,T}$ be such that $\tau\leq \theta$ $\P-$a.s. and $\theta$ is a predictable stopping time. Then
\begin{align*}
V^i(\tau,\alpha^{-i},\xi^i,\chi^i) = \esssup_{\alpha^i\in\Ac^i(\alpha^{-i})} \E_\tau^{\P^{\alpha}} \bigg[ &  e^{-\int_\tau^\theta \rho^i_s(X_s,\chi_s, \alpha_s) \mathrm{d}s} ~V^i(\theta,\alpha^{-i},\xi^i,\chi^i) \\
& + \int_\tau^\theta e^{-\int_\tau^s \rho^i_u(X_u,\chi_u,\alpha_u) \mathrm{d}u} \big( u^i_A(\chi_s) - c^i_s(X_s,\alpha_s) \big) \mathrm{d}s  \bigg],
\end{align*}
%\begin{equation*}
%V_{t}^{i}=  \esssup\limits_{\nu \in \Ac^i \times \Lc^i} \E_{t}^{\P^\nu}\left[\erm^{\gamma^i  \int_{t}^{\tau}c^i(\nu_s) \drm s }V^i_{\tau}\right]. 
%\end{equation*}
\end{proposition}

\begin{proof}
To simplify the notation, we just write $V^i_\tau = V^i(\tau,\alpha^{-i},\xi^i,\chi^i)$ and 
\[
\tilde V^i_\tau = \esssup_{\alpha^i\in\Ac^i(\alpha^{-i})} \E_\tau^{\P^{\alpha}} \bigg[   e^{-\int_\tau^\theta \rho^i_s(X_s,\chi_s, \alpha_s) \mathrm{d}s} ~V^i_\theta + \int_\tau^\theta e^{-\int_\tau^s \rho^i_u(X_u,\chi_u,\alpha_u) \mathrm{d}u} \big( u^i_A(\chi_s) - c^i_s(X_s,\alpha_s) \big) \mathrm{d}s  \bigg]. 
\]

$(i)$ From Equation \eqref{continuationutility1} and by the tower property we have 
\begin{align*}
V^i_\tau = \esssup_{\alpha^i\in\Ac^i(\alpha^{-i})} \E_\tau^{\P^{\alpha}} \bigg[   e^{-\int_\tau^\theta \rho^i_s(X_s,\chi_s, \alpha_s) \mathrm{d}s} ~L^i(\theta,\alpha^{-i},\alpha^i)  + \int_\tau^\theta e^{-\int_\tau^s \rho^i_u(X_u,\chi_u,\alpha_u) \mathrm{d}u} \big( u^i_A(\chi_s) - c^i_s(X_s,\alpha_s) \big) \mathrm{d}s  \bigg].
\end{align*}

%Take now any $\epsilon \in \Ac^i(\alpha^{-i})$. From Bayes' formula and by noticing that the quotient $\frac{M_{T}^{\epsilon}}{M_{\tau}^{\epsilon} }$ does not depend on the values of $\epsilon$ before time $\tau$, we have 
%\begin{align*}
%\E_{\theta}^{\P^\epsilon}\left[-\frac{1}{\gamma^{i}}\erm^{\gamma^{i} \left(-\xi^{i}+ \int_{\tau}^{T}c^i(\epsilon_s)\drm s\right)}\right] & 
%= \E^\P_{\theta}\left[- \frac{M_{T}^{\epsilon} }{M_{\tau}^{\epsilon}}\frac{1}{\gamma^{i}}\erm^{\gamma^{i} \left(-\xi^{i} + \int_{\tau}^{T}c^i(\epsilon_s)\drm s\right)}\right] 
%\leq \esssup_{\alpha^i\in\Ac^i(\alpha^{-i})} L^i(\tau,\alpha^{-i},\alpha^i) = V_{\tau}^{i}. 
%\end{align*}
Hence it follows $V^i_\tau  \leq \tilde V^i_\tau$, since by definition we have $\P$-a.s. that $L^i(\theta,\alpha^{-i},\alpha^i) \leq V^i_\theta$.

\medskip

$(ii)$  Let $\alpha^i, \epsilon^i\in\Ac^i(\alpha^{-i})$ and define the control $\widetilde \alpha^i_u = \alpha^i_u \mathds{1}_{0 \leq u < \theta} + \epsilon^i_u \mathds{1}_{\theta \leq u \leq T}$, for $u \in [0,T]$. Note that $\widetilde \alpha^i$ naturally belongs to $\Ac^i(\alpha^{-i})$.\footnote{Indeed, the martingale property of $M^\alpha$ is preserved under concatenation. Since $\theta$ is predictable, so is $\widetilde \alpha^i$.} By the tower property we have, denoting $\tilde\alpha=\tilde\alpha^i \otimes_i \alpha^{-i}$
\begin{align*}
V_{\tau}^{i} & \geq L^i(\tau,\alpha^{-i},\tilde\alpha^i) \\
&  = \E_\tau^{\P^{\tilde\alpha}} \bigg[   e^{-\int_\tau^\theta \rho^i_s(X_s,\chi_s, \tilde\alpha_s) \mathrm{d}s} ~L^i(\theta,\alpha^{-i},\tilde\alpha^i)  + \int_\tau^\theta e^{-\int_\tau^s \rho^i_u(X_u,\chi_u,\tilde\alpha_u) \mathrm{d}u} \big( u^i_A(\chi_s) - c^i_s(X_s,\tilde\alpha_s) \big) \mathrm{d}s  \bigg]. 
\end{align*}
Note that the quotient $M_{T}^{{\tilde \alpha}}/M_{\theta}^{{\tilde \alpha}}$ does not depend on the values of $\tilde\alpha$ before $\theta$. In particular, $L^i(\theta,\alpha^{-i},\tilde\alpha^i)$ does not depend on the values of $\tilde\alpha$ before $\theta$ and we have $L^i(\theta,\alpha^{-i},\tilde\alpha^i)=L^i(\theta,\alpha^{-i},\epsilon^i)$. It follows therefore
\begin{align*}
V_{\tau}^{i}  &\geq  \E_\tau^{\P} \bigg[ \E^{\P}_\theta \bigg[ \frac{M^{\tilde\alpha}_T}{M^{\tilde\alpha}_\tau} \bigg(e^{-\int_\tau^\theta \rho^i_s(X_s,\chi_s, \alpha_s) \mathrm{d}s} ~L^i(\theta,\alpha^{-i},\epsilon^i)  + \int_\tau^\theta e^{-\int_\tau^s \rho^i_u(X_u,\chi_u,\alpha_u) \mathrm{d}u} \big( u^i_A(\chi_s) - c^i_s(X_s,\alpha_s) \big) \mathrm{d}s \bigg) \bigg]\bigg] \\
& = \E_\tau^{\P} \bigg[  \frac{M^{\tilde\alpha}_\theta}{M^{\tilde\alpha}_\tau} \bigg(e^{-\int_\tau^\theta \rho^i_s(X_s,\chi_s, \alpha_s) \mathrm{d}s} ~L^i(\theta,\alpha^{-i},\epsilon^i)  + \int_\tau^\theta e^{-\int_\tau^s \rho^i_u(X_u,\chi_u,\alpha_u) \mathrm{d}u} \big( u^i_A(\chi_s) - c^i_s(X_s,\alpha_s) \big) \mathrm{d}s \bigg) \bigg] \\
& = \E_\tau^{\P} \bigg[  \frac{M^{\alpha}_\theta}{M^{\alpha}_\tau} \bigg(e^{-\int_\tau^\theta \rho^i_s(X_s,\chi_s, \alpha_s) \mathrm{d}s} ~L^i(\theta,\alpha^{-i},\epsilon^i)  + \int_\tau^\theta e^{-\int_\tau^s \rho^i_u(X_u,\chi_u,\alpha_u) \mathrm{d}u} \big( u^i_A(\chi_s) - c^i_s(X_s,\alpha_s) \big) \mathrm{d}s \bigg) \bigg] \\
& =  \E_\tau^{\P^{\alpha}} \bigg[  e^{-\int_\tau^\theta \rho^i_s(X_s,\chi_s, \alpha_s) \mathrm{d}s} ~L^i(\theta,\alpha^{-i},\epsilon^i)  + \int_\tau^\theta e^{-\int_\tau^s \rho^i_u(X_u,\chi_u,\alpha_u) \mathrm{d}u} \big( u^i_A(\chi_s) - c^i_s(X_s,\alpha_s) \big) \mathrm{d}s  \bigg].
\end{align*}
To conclude, note that the family $\left(L^i(\theta,\alpha^{-i},\epsilon^i) \right)_{\epsilon^i \in \Ac^i(\alpha^{-i})}$ is directed upwards. Indeed, for any $\epsilon_1$ and $\epsilon_2$ in $\Ac^i(\alpha^{-i})$, define 
\begin{equation*}
\epsilon_3 := \epsilon_1\mathds{1}_{L^i(\theta,\alpha^{-i},\epsilon_1) \geq L^i(\theta,\alpha^{-i},\epsilon_2)} + \epsilon_2\mathds{1}_{L^i(\theta,\alpha^{-i},\epsilon_1) < L^i(\theta,\alpha^{-i},\epsilon_2)}.
\end{equation*}
 By definition, $\epsilon_3$ belongs to $\Ac^i(\alpha^{-i})$ and we also have $L^i(\theta,\alpha^{-i},\epsilon_3) \geq  \max\left\{L^i(\theta,\alpha^{-i},\epsilon_1), L^i(\theta,\alpha^{-i},\epsilon_2)\right\}$. From \cite{neveu1972}, Proposition VI.1.1, there exists a sequence $(\epsilon_n)_{n \in \N}$ in $\Ac^i(\alpha^{-i})$ such that $\left(L^i(\tau,\alpha^{-i}, \epsilon_{n})\right)_{n \in \N}$ is non-decreasing almost surely and $\esssup\limits_{\epsilon \in \Ac^i(\alpha^{-i})} L^i(\tau,\alpha^{-i}, \epsilon) = \lim\limits_{n \rightarrow \infty} L^i(\tau,\alpha^{-i}, \epsilon_{n})$ almost surely. We have finally from the monotone convergence theorem
\begin{align*}
V_{\tau}^{i} & \geq \lim\limits_{n \rightarrow \infty} \E_\tau^{\P^{\alpha}} \bigg[  e^{-\int_\tau^\theta \rho^i_s(X_s,\chi_s, \alpha_s) \mathrm{d}s} ~L^i(\theta,\alpha^{-i},\epsilon_n)  + \int_\tau^\theta e^{-\int_\tau^s \rho^i_u(X_u,\chi_u,\alpha_u) \mathrm{d}u} \big( u^i_A(\chi_s) - c^i_s(X_s,\alpha_s) \big) \mathrm{d}s  \bigg] \\
&  = \E_\tau^{\P^{\alpha}} \bigg[ \lim\limits_{n \rightarrow \infty}  e^{-\int_\tau^\theta \rho^i_s(X_s,\chi_s, \alpha_s) \mathrm{d}s} ~L^i(\theta,\alpha^{-i},\epsilon_n)  + \int_\tau^\theta e^{-\int_\tau^s \rho^i_u(X_u,\chi_u,\alpha_u) \mathrm{d}u} \big( u^i_A(\chi_s) - c^i_s(X_s,\alpha_s) \big) \mathrm{d}s  \bigg] = \tilde V_\tau^i.
\end{align*}

\end{proof}

\subsection{On the martingale representation decomposition} \label{app:MRP-coordinates}

In this section, we let $X=((X^1)^\top,\dots,(X^N)^\top)^\top$ and we assume the $X^i$'s do not jump simultaneously. For every $i\in I$, we denote by $T^i_k$ the $k-$th jump time of the process $X^i$ and we define the jump times of $X$ as follows
\[
T_1 =  \inf \{ T_1^i: i=1,\dots,N \}, \quad T_n= \inf\{ T_k^i: T_{n-1} < T_k^i, \ i=1,\dots,d, \ k\geq 1 \}.
\]
We define also $T_\infty := \lim_{n\to\infty} T_n$.
\begin{lemma}\label{lemma:MRP-coordinates}
Suppose $\P(T_\infty < \infty)=0$. Then for any $H\in \L_\text{loc}^1(\mu_X)$ there exist $H^1,\dots,H^N$, with $H^i\in\L_\text{loc}^1(\mu_{X^i})$, such that
\[
\int_0^t \int_{\R^{dN}\setminus\{0\}} H_s(u) \mu_X(\emph{d}s,\emph{d}u) =  \sum_{i=1}^N \int_0^t \int_{\R^d\setminus\{0\}} H_s^i(u) \mu_{X^i}(\emph{d}s,\emph{d}u).
\]
\end{lemma}

\begin{proof}
$(i)$ Notice that the measure $\mu_X$ is supported over the `axis-blocks' of $\R^{dN}$. To be more precise, since each $X^i$ jumps independently we have
\[
\mu_X\big((0,t]\times A \big) = \sum_{i=1}^N \mu_{X^i} \big( (0,t]\times A^i \big), \quad \text{for every } t\in[0,T], \ A\in\Bc(\R^{dN}\setminus\{0\}),
\]
where each set $A^i$ is the projection over $\R^d$ of the set $A \cap ( \R^d \otimes_i 0^{d(N-1)} )$. Indeed, we have
\begin{align*}
\mu_X\big((0,t]\times A \big)  & = \sum_{n\leq 1} {\bf 1}_{\{T_n \geq t\}}   {\bf 1}_{\{\Delta X_{T_n} \in A\}}  = \sum_{n\geq 1} {\bf 1}_{\{T_n \leq t\}}  \sum_{i=1}^N \sum_{k\geq 1} {\bf 1}_{\{\Delta X_{T_n}^i \in A^i\}} {\bf 1}_{\{ T_n=T^i_k\}} \\
& = \sum_{n\geq 1}  \sum_{i=1}^N \sum_{k\geq 1} {\bf 1}_{\{T_n \leq t\}}  {\bf 1}_{\{\Delta X_{T_n}^i \in A^i\}} {\bf 1}_{\{ T_n=T^i_k\}}  = \sum_{n\geq 1}  \sum_{i=1}^N \sum_{k\geq 1} {\bf 1}_{\{T_k^i \leq t\}}  {\bf 1}_{\{\Delta X_{T_k^i}^i \in A^i\}} {\bf 1}_{\{ T_n=T^i_k\}}  \\
& =\sum_{i=1}^N   \sum_{k\geq 1}  {\bf 1}_{\{T_k^i \leq t\}}  {\bf 1}_{\{\Delta X_{T_k^i}^i \in A^i\}}   \sum_{n\geq 1}  {\bf 1}_{\{ T_n=T^i_k\}} =\sum_{i=1}^N   \sum_{k\geq 1}  {\bf 1}_{\{T_k^i \leq t\}}  {\bf 1}_{\{\Delta X_{T_k^i}^i \in A^i\}}  \\
& = \sum_{i=1}^N \mu_{X^i}\big( (0,t]\times A^i \big),
\end{align*}
where we used, since $\P(T_\infty < \infty)=0$, that for every $i\in I$, $k\geq 1$ we have $\sum_{n\geq 1}  {\bf 1}_{\{ T_n=T^i_k\}} =1$.

\medskip
$(ii)$ We can decompose then the integral
\begin{align*}
\int_0^t \int_{\R^{dN} \setminus\{0\}} H_s(u) \mu_X(\textrm{d}s,\textrm{d}u) & =  \sum_{i=1}^N \int_0^t \int_{ \R^d \setminus\{0\} \otimes_i 0^{d(N-1)}} H_s(u) \mu_X(\textrm{d}s,\textrm{d}u) \\
& =  \sum_{i=1}^N \int_0^t \int_{ \R^d \setminus\{0\} \otimes_i 0^{d(N-1)}} H_s(u^i \otimes_i 0^{d(N-1)}) \mu_X(\textrm{d}s,\textrm{d}u) \\
& =  \sum_{i=1}^N \int_0^t \int_{\R^d\setminus\{0\}} H_s^i (u^i)  \int_{ 0^{d(N-1)}}  \mu_X(\textrm{d}s,\textrm{d}u) \\
& =  \sum_{i=1}^N \int_0^t \int_{\R^d\setminus\{0\}} H_s^i (v) \mu_{X^i} (\textrm{d}s,\textrm{d}v),
\end{align*}
with the predictable functions $H_s^i(v):= H_s(v\otimes_i 0^{d(N-1)})$.

\end{proof}

\section{Proofs for the agents' game} \label{app:proofs-agents}

\begin{proof}[\bf Proof of Proposition \ref{prop-agent-bsde} ]
$(i)$ Let $\alpha^\star\in\text{NE}(\xi,\chi)$ and fix the agent $i\in I$. The action $\alpha^{\star,i}$ maximizes his utility, given the actions of the others, that is
\[
V_0^i(\alpha^{\star,-i},\xi^i,\chi^i) = U_0^i(\alpha^{\star,i},\alpha^{\star,-i},\xi^i,\chi^i).
\]
We recall the family of random variables $V^i(\tau,\alpha^{\star,-i},\xi^i,\chi^i)$ defined in \eqref{eq:dynamic-value-agent}. By Definition \ref{def-admisible-contracts} (i), it follows from the admissibility of the contract that for any $\alpha^i \in\Ac^i(\alpha^{\star,-i})$ we can find $p_i=p_i(\alpha^i\otimes_i\alpha^{\star,-i})>1$ such that
\begin{equation}\label{eq:V-tau-in-Lq}
\sup_{\tau\in\Tc_{0,T}}  \E^{\P^{\alpha^i\otimes_i\alpha^{\star,-i} }}\big[| V^i(\tau,\alpha^{\star,-i},\xi^i,\chi^i)  |^{p_i} \big]  < \infty.
\end{equation}

Next, we have the Dynamic Programming Principle\footnote{see Proposition \ref{prop-DPP}, or the more general result for the case without jumps \cite[Theorem 3.4]{karoui2013capacities2}.},
for any predictable $\theta\in\Tc_{0,T}$ such that $\tau\leq \theta$ it holds $\P-$a.s.
\begin{align*}
V^i(\tau,\alpha^{\star,-i},\xi^i,\chi^i) = \esssup_{\alpha^i\in\Ac^i(\alpha^{\star,-i})} \E_\tau^{\P^{\alpha^i\otimes_i\alpha^{\star,-i} }} \bigg[ &  e^{-\int_\tau^\theta \rho^i_s(X_s,\chi_s,\hat \alpha_s) \mathrm{d}s} ~V^i(\theta,\alpha^{\star,-i},\xi^i,\chi^i) \\
& + \int_\tau^\theta e^{-\int_\tau^s \rho^i_u(X_u,\chi_u,\hat\alpha_u) \mathrm{d}u} \big( u^i_A(\chi_s) - c^i_s(X_s,\hat\alpha_s) \big) \mathrm{d}s  \bigg],
\end{align*}
with the notation $\hat \alpha := \alpha^i \otimes_i \alpha^{\star,-i} $. Fix now $\alpha^i\in\Ac^i(\alpha^{\star,-i})$ and define the family, for $\tau\in\Tc_{0,T}$
\[
R^{\alpha^i}(\tau) := V^i(\tau,\alpha^{\star,-i},\xi^i,\chi^i) ~  e^{- \int_0^\tau \rho^i_s(X_s,\chi_s,\hat\alpha_s) \mathrm{d}s} + \int_0^\tau e^{-\int_0^s \rho^i_u(X_u,\chi_u,\hat\alpha_u) \mathrm{d}u} \big( u^i_A(\chi_s) - c^i_s(X_s,\alpha_s) \big) \mathrm{d}s.
\]
It follows that the family $(R^{\alpha^i}(\tau))_{\tau\in\Tc_{0,T}}$ is a $\P^{\alpha^i\otimes_i\alpha^{\star,-i} }-$supermartingale\footnote{The family is $\P^{\alpha^i\otimes_i\alpha^{\star,-i} }-$integrable from \eqref{eq:V-tau-in-Lq}, H\"older's inequality and Definition \ref{def-admisible-contracts} (ii)-(iii). } system (see \cite[Definition 10]{della1981sur} or  \cite[Section 3.3]{bouchard2016}). Consequently, by \cite[Theorem 15]{della1981sur}, $(R^{\alpha^i}(\tau))_{\tau\in\Tc_{0,T}}$ can be aggregated in a unique (up to indistinguishability) optional process which is given by
\[
R_t^{\alpha^i}: = V^i_t(\alpha^{\star,-i},\xi^i,\chi^i) ~  e^{- \int_0^t \rho^i_s(X_s,\chi_s,\hat\alpha_s) \mathrm{d}s} + \int_0^t e^{-\int_0^s \rho^i_u(X_u,\chi_u,\hat\alpha_u) \mathrm{d}u} \big( u^i_A(\chi_s) - c^i_s(X_s,\hat\alpha_s) \big) \mathrm{d}s,
\]
where we have $\P-$a.s.
\[
V^i_t(\alpha^{\star,-i},\xi^i,\chi^i)  = \esssup_{\alpha^i\in\Ac^i(\alpha^{\star,-i})} \E_t^{\P^{\hat\alpha }} \left[ e^{-\int_t^T \rho^i_s(X_s,\chi_s,\hat\alpha_s) \mathrm{d}s} ~ \Uc^i_A (\xi^i) + \int_t^T e^{-\int_t^s \rho^i_u(X_u,\chi_u,\hat\alpha_u) \mathrm{d}u} \big( u^i_A(\chi_s) - c^i_s(X_s,\hat\alpha_s) \big) \mathrm{d}s \right].
\]
We deduce that $\forall\alpha^i\in\Ac^i(\alpha^{\star,-i})$, the process $R^{\alpha^i}$ is a $\P^{\alpha^i\otimes_i\alpha^{\star,-i} }-$supermartingale. Note that $R_0^{\alpha^i}=V_0^i(\alpha^{\star,-i},\xi^i,\chi^i)$. Therefore, for any $t\in[0,T]$
\begin{align*}
V_0^i(\alpha^{\star,-i},\xi^i,\chi^i)= R_0^{\alpha^{\star,i}} & \geq  \E^{\P^{\alpha^\star}} \big[ R_t^{\alpha^{\star,i}}\big]  \geq \E^{\P^{\alpha^\star}} \big[ R_T^{\alpha^{\star,i}} \big]  =U_0^i(\alpha^{\star,i},\alpha^{\star,-i},\xi^i,\chi^i) = V_0^i(\alpha^{\star,-i},\xi^i,\chi^i),
\end{align*}
and it follows that $R^{\alpha^{\star,i}}$ is a $\P^{\alpha^\star}-$martingale. From the martingale representation theorem\footnote{Notice that $\P^{\alpha^{\star}}$ satisfies also the martingale representation property, due to Theorem III.5.24 in \cite{jacod2003limit}. %Moreover, by Lemma \ref{lemma:MRP-coordinates}, the representation can be done with respect to the measures $\mu_{X^i}$.}
}, there exists processes $Z^i \in \H^{2,1\times dN}_\text{loc}(X)$ and $H^i=(H^{i,1},\dots,H^{i,N})$ with each $H^{i,\ell} \in \L^{1}_{\text{loc}}(\mu_{X^\ell})$ such that it holds $\P^{\alpha^\star}$-a.s.
\[
\mathrm{d} R^{\alpha^{\star,i}}_t =  e^{- \int_0^t \rho^i_s(X_s,\chi_s,\alpha_s^\star) \mathrm{d}s} \bigg(  Z^i_t \mathrm{d} X_t^{c,\alpha^\star} +  \sum_{\ell=1}^N \int_{\R^{d}\setminus \{0\}} H^{i,\ell}_t(x)  \big( \mu_{X^\ell}(\textrm{d}t,\textrm{d}x) - \nu_t^{\ell,\alpha^\star}(\textrm{d}x)\textrm{d}t   \big) \bigg).
\]
By applying It\^o's formula we obtain
\begin{align*}
\mathrm{d}V^i_t(\alpha^{\star,-i},\xi^i,\chi^i) & = \big(  \rho^i_t(X_t,\chi_t,\alpha^\star_t) V^i_t(\alpha^{\star,-i},\xi^i,\chi^i) -u^i_A(\chi_t) + c^i_t(X_t,\alpha^\star_t) - Z^i_t \Sigma_t(X_t)b_t(X_t,\alpha^\star_t) \big) \mathrm{d}t + Z^i_t \mathrm{d}X_t^c \\
&  \hspace{.3cm} +  \sum_{\ell=1}^N \int_{\R^{d}\setminus \{0\}} H^{i,\ell}_t(x)  \big( \mu_{X^\ell}(\textrm{d}t,\textrm{d}x) - \nu_t^{\ell,\alpha^\star}(\textrm{d}x)\textrm{d}t   \big), \quad \P-\text{a.s.}
\end{align*}
From now on write $Y^i_t:= V^i_t(\alpha^{\star,-i},\xi^i,\chi^i)$ and define $Y:=(Y^1,\dots,Y^N)^\top$, $Z:=((Z^1)^\top,\dots,(Z^N)^\top)^\top$, $H:=((H^1)^\top,\dots,(H^N)^\top)^\top$. Then we have that $Y^i_T = \Uc^i_A(\xi^i)$ and
\[
\mathrm{d}Y^i_t = -f^i_t(X_t,Y_t,Z_t,H_t,\chi_s,\alpha^\star_t) \mathrm{d}t + Z^i_t \mathrm{d}X^c_t + \int_{\R^{d}\setminus \{0\}} H^{i}_t(x)\vec{\mu}_X(\textrm{d}t,\textrm{d}x) , \quad \P-\text{a.s.}
\]
Next, note that for any $\alpha^i\in\Ac^i(\alpha^{\star,-i})$ 
%\begin{align*}
% R_t^{\alpha^i}  = &  Y^i_0 + \int_0^t  e^{- \int_0^s \rho^i_u(X_u,\chi_u,\hat\alpha_u) \mathrm{d}u} (\mathrm{d}Y^i_s - \rho^i_s(X_s,\chi_s,\hat\alpha_s)Y^i_s \mathrm{d}s) + \int_0^t e^{-\int_0^s \rho^i_u(X_u,\chi_u,\hat\alpha_u) \mathrm{d}u} \big( u^i_A(\chi_s) - c^i_s(X_s,\hat\alpha_s) \big) \mathrm{d}s\\
%  = & Y^i_0 +  \int_0^t e^{- \int_0^s \rho^i_u(X_u,\chi_u,\hat\alpha_u) \mathrm{d}u} \bigg(  \big( \rho^i_s(X_s,\chi_s,\alpha_s^\star)-\rho^i_s(X_s,\chi_s,\hat\alpha_s) \big) Y^i_s + c^i_s(X_s,\alpha^\star_s) - c^i_s(X_s,\hat\alpha_s)  \bigg) \mathrm{d}s   \\
% &   -  \int_0^t e^{- \int_0^s \rho^i_u(X_u,\chi_u,\hat\alpha_u) \mathrm{d}u}  Z^i_s\Sigma_s(X_s)b_s(X_s,\alpha^\star_s) \mathrm{d}s  +  \int_0^t e^{-\int_0^s \rho^i_u(X_u,\chi_u,\hat\alpha_u) \mathrm{d}u} Z^i_s  \mathrm{d}X_s \\
% = & Y^i_0 + \int_0^t e^{- \int_0^s \rho^i_u(X_u,\chi_u,\hat\alpha_u) \mathrm{d}u}   \big( f^i_s(X_s,Y_s,Z_s,\chi_s,\hat\alpha_s) - f^i_s(X_s,Y_s,Z_s,\chi_s,\alpha^\star_s) \big) \mathrm{d}s  \\ 
% & +  \int_0^t e^{-\int_0^s \rho^i_u(X_u,\chi_u,\hat\alpha_u) \mathrm{d}u} Z^i_s\Sigma_s(X_s)\mathrm{d}W_s^{\hat\alpha}
%\end{align*}
\begin{align*}
 \textrm{d} R_s^{\alpha^i}  = &   e^{- \int_0^s \rho^i_u(X_u,\chi_u,\hat\alpha_u) \mathrm{d}u} (\mathrm{d}Y^i_s - \rho^i_s(X_s,\chi_s,\hat\alpha_s)Y^i_s \mathrm{d}s) + e^{-\int_0^s \rho^i_u(X_u,\chi_u,\hat\alpha_u) \mathrm{d}u} \big( u^i_A(\chi_s) - c^i_s(X_s,\hat\alpha_s) \big) \mathrm{d}s\\
 = &  e^{- \int_0^s \rho^i_u(X_u,\chi_u,\hat\alpha_u) \mathrm{d}u} \big(  ( \rho^i_s(X_s,\chi_s,\alpha_s^\star)-\rho^i_s(X_s,\chi_s,\hat\alpha_s) ) Y^i_s + c^i_s(X_s,\alpha^\star_s) - c^i_s(X_s,\hat\alpha_s)  \big) \mathrm{d}s   \\
 &   + e^{- \int_0^s \rho^i_u(X_u,\chi_u,\hat\alpha_u) \mathrm{d}u} \bigg( Z^i_s  \mathrm{d}X^c_s -  Z^i_s\Sigma_s(X_s)b_s(X_s,\alpha^\star_s) \mathrm{d}s  + \int_{\R^d\setminus\{0\}} H_t^i(x) (\vec{\mu}_X (\textrm{d}t,\textrm{d} x) - \vec{\nu}^{\alpha^\star}_t(\textrm{d}x) \textrm{d}t  )\bigg) \\
 = & e^{- \int_0^s \rho^i_u(X_u,\chi_u,\hat\alpha_u) \mathrm{d}u}   \big( f^i_s(X_s,Y_s,Z_s,\chi_s,\hat\alpha_s) - f^i_s(X_s,Y_s,Z_s,\chi_s,\alpha^\star_s) \big) \mathrm{d}s  \\ 
 &  e^{-\int_0^s \rho^i_u(X_u,\chi_u,\hat\alpha_u) \mathrm{d}u} Z^i_s\Sigma_s(X_s)\mathrm{d}W_s^{\hat\alpha} +  e^{-\int_0^s \rho^i_u(X_u,\chi_u,\hat\alpha_u) \mathrm{d}u} \bigg(  \int_{\R^d\setminus\{0\}} H_t^i(x) (\vec{\mu}_X (\textrm{d}t,\textrm{d} x) - \vec{\nu}^{\hat\alpha}_t(\textrm{d}x) \textrm{d}t  ) \bigg)
\end{align*}

Since $R_t^{\alpha^i}$ is a $\P^{\alpha^i\otimes_i \alpha^{\star,-i} }$--supermartingale, we conclude that $\alpha^\star$ satisfies \eqref{eq:a-star} and $(Y,Z,H)$ is a solution to BSDE \eqref{eq:bsde-agent}.

\medskip
To conclude, let us prove that $(Z,H)\in\Vc^{Y_0,\chi}$. Since $Y^i = Y^{Y_0,Z,H,\chi,i}$, with $Y^i_0\in\R$, it follows from \eqref{eq:V-tau-in-Lq} that $Y^{Y_0,Z,H,\chi}$ has the required integrability, by taking $q_{\alpha^i}=p_i$.

\medskip
$(ii)$ Let $(Y,Z,H)$ be a solution to BSDE \eqref{eq:bsde-agent} such that $(Z,H)\in\Vc^{Y_0,\chi}$ and define
\[
\alpha^\star_s = a^\star(s,X_s,Y_s,Z_s,H_s,\chi_s), \quad \text{d}t\otimes\text{d}\P-\text{a.s. over } [0,T]\times \Omega.
\] 
Then we have $\alpha^\star\in\Ac$. Fix $i\in I$, since $Y^i_T=\Uc^i_A(\xi^i)$, for any $\alpha^i\in\Ac^i(\alpha^{\star,-i})$ we have by It\^o's formula
\begin{align*}
U_0^i(\alpha^i,\alpha^{\star,-i},\xi^i,\chi^i) & = \E^{\P^{\alpha^i\otimes_i\alpha^{\star,-i} }} \left[ e^{-\int_0^T \rho^i_s(X_s,\chi_s,\hat\alpha_s) \mathrm{d}s} ~ Y^i_T + \int_0^T e^{-\int_0^s \rho^i_u(X_u,\chi_u,\hat\alpha_u) \mathrm{d}u} \big( u^i_A(\chi_s) - c^i_s(X_s,\hat\alpha_s) \big) \mathrm{d}s  \right] \\
& = \E^{\P^{\alpha^i\otimes_i\alpha^{\star,-i} }} \bigg[ Y^i_0 + \int_0^T e^{-\int_0^s \rho^i_u(X_u,\chi_u,\hat\alpha_u) \mathrm{d}u} \left( f_s^i(X_s,Y_s,Z_s,\chi_s,\hat\alpha_s)  - F^i_s(X_s,Y_s,Z_s,\chi_s) \right) \mathrm{d}s   \bigg]\\
& ~~ + \E^{\P^{\alpha^i\otimes_i\alpha^{\star,-i} }} \bigg[ \int_0^T e^{-\int_0^s \rho^i_u(X_u,\chi_u,\hat\alpha_u) \mathrm{d}u}  \bigg(   Z^{i,:}_s   \mathrm{d}X^{c,\hat\alpha}_s +    \int_{\R^d\setminus\{0\}} H_t^{i,:}(x) (\vec{\mu}_X (\textrm{d}t,\textrm{d} x) - \vec{\nu}^{\hat\alpha}_t(\textrm{d}x) \textrm{d}t  ) \bigg)  \bigg] 
\end{align*}
Notice that $Y^i = Y^{Y_0,Z,H,\chi,i}$, with $Y^i_0\in\R$.  By the definition of the class $\Vc^{Y_0,\chi}$, it follows that $Y^i$ is a process of class (D) under $\P^{\alpha^i\otimes_i\alpha^{\star,-i} }$. From Definition \ref{def-admisible-contracts} $(ii)$, so is the process $ Y e^{-\int_0^\cdot \rho^i_s(X_s,\chi_s,\hat\alpha_s) \mathrm{d}s}$, which implies that the second term above is a $\P^{\alpha^i\otimes_i\alpha^{\star,-i} }-$martingale. We have thus
\[
U_0^i(\alpha^i,\alpha^{\star,-i},\xi^i,\chi^i) = Y^i_0 + \E^{\P^{\alpha^i\otimes_i\alpha^{\star,-i} }} \bigg[  \int_0^T e^{-\int_0^s \rho^i_u(X_u,\chi_u,\hat\alpha_u) \mathrm{d}u} \left( f_s^i(X_s,Y_s,Z_s,\chi_s,\hat\alpha_s)  - F^i_s(X_s,Y_s,Z_s,\chi_s) \right) \mathrm{d}s   \bigg].
\]
Therefore, $U_0^i(\alpha^i,\alpha^{\star,-i},\xi^i,\chi^i) \leq Y^i_0$ for every $\alpha^i\in\Ac^i(\alpha^{\star,-i})$ and $U_0^i(\alpha^{\star,i},\alpha^{\star,-i},\xi^i,\chi^i)= Y_0^i$. We conclude that $\alpha^\star\in\text{NE}(\xi,\chi)$.
\end{proof}

\bigskip
\begin{proof}[\bf Proof of Corollary \ref{cor-optimal-response}]
The fact that $a^{\star,y,Z,H,\chi}$ is a Nash equilibrium for the contract $(\xi,\chi)$ follows from Proposition \ref{prop-agent-bsde}, by noting that $(Y^{y,Z,H,\chi},Z,H)$ is a solution to BSDE \eqref{eq:bsde-agent}. Moreover, from part $(ii)$ of the proof of the proposition we obtain that $V_0^i(\alpha^{\star,-i},\xi^i,\chi^i)=Y_0^{y,Z,H,\chi,i}=y^i$.
\end{proof}

\bigskip
\begin{proof}[\bf Proof of Proposition \ref{prop-agent-bsde-cert-equiv}]
$(i)$ Let $\alpha^\star\in\text{NE}(\xi,\chi)$ and fix the agent $i\in I$. By repeating the arguments from Proposition \ref{prop-agent-bsde}, for every $\alpha^i \in\Ac^i(\alpha^{\star,-i})$ we can find $p_i=p_i(\alpha^i\otimes_i\alpha^{\star,-i})>1$ such that
\begin{equation}\label{eq:V-tau-in-Lq-cert-equiv}
\sup_{\tau\in\Tc_{0,T}}  \E^{\P^{\alpha^i\otimes_i\alpha^{\star,-i} }} [|V^i(\tau,\alpha^{\star,-i},\xi^i,\chi^i)|^{p_i} ]  < \infty.
\end{equation}
It follows from the Dynamic Programming Principle that, for every $\alpha^i \in\Ac^i(\alpha^{\star,-i})$, the family 
\[
\bigg(  V^i(\tau,\alpha^{\star,-i},\xi^i,\chi^i) ~  e^{- \int_0^\tau \rho^i_s(X_s,\chi_s,\hat\alpha_s) \mathrm{d}s}  \bigg)_{\tau\in\Tc_{0,T}}
\]
%is a $\P^{\alpha}-$supermartingale system (see section 3.3 in \cite{bouchard2016}). Moreover, de la Vall\'ee-Poussin criterion ensures that the system is uniformly integrable  and therefore, from the results in \cite{della1981sur} (or Theorem 3.1 in \cite{bouchard2016}), it 
can be aggregated in a unique (up to indistinguishability) optional process which is given by
\[
R_t^{\alpha^i} : = V_t^i(\alpha^{\star,-i},\xi^i,\chi^i) ~  e^{- \int_0^t \rho^i_s(X_s,\chi_s,\hat\alpha_s) \mathrm{d}s},
\]
where
\[
V_t^i(\alpha^{\star,-i},\xi^i,\chi^i)= \esssup_{ \alpha^i \in\Ac^i(\alpha^{\star,-i}) }   \E^{\P^{\alpha^i\otimes_i\alpha^{\star,-i} }} \left[ e^{-\int_t^T \rho^i_s(X_s,\chi_s,\hat\alpha_s) \mathrm{d}s} ~ \Uc^i_A (\xi^i) ~ \bigg| ~ \Fc_t  \right], ~\P\text{\rm--a.s.}
\]
Moreover, $\forall \alpha^i \in\Ac^i(\alpha^{\star,-i})$ the process $R^{\alpha^i}$ is a $\P^{\alpha^i \otimes_i \alpha^{\star,-i}}-$supermartingale and $R^{\alpha^{\star,i}}$ is a $\P^{\alpha^\star}-$martingale. From the multiplicative decomposition of negative martingales (see for instance \cite[Equation (8)]{yoeurp1976decomposition}) and the martingale representation theorem, there exists predictable processes $Z^i \in \H^{2,1\times dN}_\text{loc}(X)$ and and $\tilde H^i=(\tilde H^{i,1},\dots,\tilde H^{i,N})$ with each $\tilde H^{i,\ell}\in \L^{1}_{\text{loc}}(\mu_{X^\ell})$ such that
\[
R^{\alpha^{\star,i}}_t = V_0^i(\alpha^{\star,-i},\xi^i,\chi^i) \Ec\bigg( -R_A^i \int_0^t Z^i_s \mathrm{d}X^{c,\alpha^\star}_t  -R_A^i  \int_{\R^d\setminus\{0\}} \tilde H_t^i(x) (\vec{\mu}_X (\textrm{d}t,\textrm{d} x) - \vec{\nu}^{\alpha^\star}_t(\textrm{d}x) \textrm{d}t\bigg).
\]
By applying It\^o's formula to $R^{\alpha^{\star,i}}$ we obtain $\P-\text{a.s.}$
\begin{align*}
\mathrm{d} V_t^i(\alpha^{\star,-i},\xi^i,\chi^i) = & ~ V_t^i(\alpha^{\star,-i},\xi^i,\chi^i) \bigg( \rho^i_t(X_t,\chi_t,\alpha^\star_t)  \mathrm{d}t  - R_A^i  Z_t^i  \mathrm{d}X^{c,\alpha^\star}_t -R_A^i  \int_{\R^d\setminus\{0\}} \tilde H_t^i(x) (\vec{\mu}_X (\textrm{d}t,\textrm{d} x) - \vec{\nu}^{\alpha^\star}_t(\textrm{d}x) \textrm{d}t\bigg),
\end{align*}
Write now $Y_t^i:= -\frac{1}{R_A^i}\log(-V_t^i(\alpha^{\star,-i},\xi^i,\chi^i) )$, $H_t^{i,\ell}(x):=\frac{1}{R_A^i}\log(1-R_A^i \tilde H_t^{i,\ell}(x))$ and define $Y:=(Y^1,\dots,Y^N)^\top$, $Z:=((Z^1)^\top,\dots,(Z^N)^\top)^\top$, $H:=((H^1)^\top,\dots,(H^N)^\top)^\top$. We have $Y^i_T = \xi^i$ and by It\^o's formula
%\[
%\mathrm{d}Y_t =\bigg( - \frac{1}{R_A}\rho_t(X_t,\alpha^\star_t)  + \sigma_t(X_t)b_t(X_t,\alpha^\star_t)\cdot Z_t  + \frac{1}{2}R_A Z_t^\top\sigma_t(X_t)\sigma^\top_t(X_t)Z_t \bigg) \mathrm{d}t - Z_t \cdot \mathrm{d}X_t,
%\]
\[
\mathrm{d}Y^i_t = -g^i_t(X_t,Z_t,H_t,\chi_t,\alpha^\star_t) \mathrm{d}t + Z^i_t \mathrm{d}X^c_t - \int_{\R^d\setminus\{0\}} H_t^{i}(x)\vec{\mu}_X(\textrm{d}t,\textrm{d} x)  , \quad \P-\text{a.s.}
\]
Next, note that for any $\alpha^i \in\Ac^i(\alpha^{\star,-i})$ it holds $R_t^{\alpha^i}  =    R_t^{\alpha^{\star,i}} e^{-\int_0^t ( \rho_u^i(X_u,\chi_u,\hat\alpha_u) - \rho_u^i(X_u,\chi_u,\alpha_u^\star)) \mathrm{d}u }$. Then we have
\begin{align*}
 \textrm{d} R_s^{\alpha^i}  = &  -R_s^{\alpha^{\star,i}} e^{-\int_0^s ( \rho_u^i(X_u,\chi_u,\hat\alpha_u) - \rho_u^i(X_u,\chi_u,\alpha_u^\star)) \mathrm{d}u } \big( \rho_s^i(X_s,\chi_s,\hat\alpha_s) - \rho_s^i(X_s,\chi_s,\alpha_s^\star)\big) \mathrm{d}s \\
 & + e^{-\int_0^s ( \rho_u^i(X_u,\chi_u,\hat\alpha_u) - \rho_u^i(X_u,\chi_u,\alpha_u^\star)) \mathrm{d}u } \big(  e^{- \int_0^s \rho_u^i(X_u,\chi_u,\alpha^\star_u) \mathrm{d}u} \mathrm{d}V_s^i(\alpha^{\star,-i},\xi^i,\chi^i) -  R_s^{\alpha^{\star,i}}\rho_s^i(X_s,\chi_s,\alpha^\star_s) \mathrm{d}s \big)  \\
=  & -  R_s^{\alpha^{\star,i}} e^{-\int_0^s ( \rho_u^i(X_u,\chi_u,\hat\alpha_u) - \rho_u^i(X_u,\chi_u,\alpha_u^\star)) \mathrm{d}u } \big( \rho_s^i(X_s,\chi_s,\hat\alpha_s) - \rho_s^i(X_s,\chi_s,\alpha_s^\star)\big) \mathrm{d}s \\
 & -   R_s^{\alpha^{\star,i}} e^{-\int_0^s ( \rho_u^i(X_u,\chi_u,\hat\alpha_u) - \rho_u^i(X_u,\chi_u,\alpha_u^\star)) \mathrm{d}u } \bigg(  R_A^i Z^i_s \mathrm{d}X^{c,\alpha^\star}_s + R_A^i\int_{\R^d\setminus\{0\}} \tilde H^i_s(x)   (\vec{\mu}_X (\textrm{d}s,\textrm{d} x) - \vec{\nu}^{\alpha^\star}_s(\textrm{d}x) \textrm{d}s \bigg) \\
 % & -   R_s^{\alpha^{\star,i}} e^{-\int_0^s ( \rho_u^i(X_u,\chi_u,\hat\alpha_u) - \rho_u^i(X_u,\chi_u,\alpha_u^\star)) \mathrm{d}u } \big( - R_A^i Z^i_s \Sigma_s(X_s)b_s(X_s,\alpha^\star_s) \mathrm{d}s + R_A^i Z^i_s \mathrm{d}X_s    \big) \\
=  & -    R_s^{\alpha^{\star,i}} e^{-\int_0^s ( \rho_u^i(X_u,\chi_u,\alpha_u) - \rho_u^i(X_u,\chi_u,\alpha_u^\star)) \mathrm{d}u }  R_A^i\big( g^i_s(X_s,Z_s,\chi_s,\hat\alpha_s)- g^i_s(X_s,Z_s,\chi_s,\alpha^\star_s) \big) \mathrm{d}s \\
   & -  R_s^{\alpha^{\star,i}} e^{-\int_0^s ( \rho_u^i(X_u,\chi_u,\alpha_u) - \rho_u^i(X_u,\chi_u,\alpha_u^\star)) \mathrm{d}u } \bigg( R_A^i  Z_s^i \Sigma_s(X_s) \mathrm{d}W^{\hat\alpha}_s + R_A^i\int_{\R^d\setminus\{0\}} \tilde H^i_s(x)(\vec{\mu}_X (\textrm{d}t,\textrm{d} x) - \vec{\nu}^{\hat\alpha}_t(\textrm{d}x) \textrm{d}t\bigg)
\end{align*}
Since $R_t^{\alpha^i}$ is a (negative) $\P^{\alpha^i \otimes_i \alpha^{\star,-i}}$--supermartingale, we have that $\alpha^\star$ satisfies \eqref{eq:a-star-cert-equiv} and $(Y,Z,H)$ is a solution to BSDE \eqref{eq:bsde-agent-cert-equiv}.

\medskip
To conclude, we prove that $(Z,H)\in\hat\Vc^\chi$. Note that $Y^{0,Z,H,\chi} = Y - Y_0$, with $Y_0\in\R^N$ and $\Uc_A^i(Y_t^i)=V_t^i(\alpha^{\star,-i},\xi^i,\chi^i)$. The required integrability follows from \eqref{eq:V-tau-in-Lq-cert-equiv} by taking $q_{\alpha^i} = p_i$.

\medskip
$(ii)$ Let $(Y,Z,H)$ be a solution to \eqref{eq:bsde-agent-cert-equiv}, with $(Z,H)\in\hat \Vc^\chi$, and let $\alpha^\star$ be a joint action satisfying \eqref{eq:a-star-cert-equiv}. Then $\alpha^\star\in\Ac$. Fix $i\in I$ and define for every $\alpha^i \in\Ac^i(\alpha^{\star,-i})$  the process 
\[
U_t^{\alpha^i}:= e^{-\int_0^t \rho^i_s(X_s,\chi_s,\hat\alpha_s) \textrm{d}s}\Uc_A^i(Y^i_t), \quad t\in[0,T],
\] 
which is of class (D) under $\P^{\alpha^i \otimes_i \alpha^{\star,-i}}$ since $(Z,H)\in\hat \Vc^\chi$. By It\^o's formula, we have
\begin{align*}
\frac{\textrm{d} U_t^{\alpha^i}}{U_t^{\alpha^i}} & = -R_A^i  \textrm{d}Y^i_t -  \rho^i_t(X_t,\chi_t,\hat\alpha_t) \textrm{d}t  + \frac{1}{2} (R_A^i)^2  || Z_t^{i,:}\Sigma_t(x)||^2 \textrm{d}t \\
& =  \bigg(  \rho^i_t(X_t,\chi_t,\alpha_t^\star) -  \rho^i_t(X_t,\chi_t,\hat\alpha_t) +R_A^i  Z_t^{i,:}  \Sigma_t(X_t) b_t(X_t,\alpha_t^\star) + \frac{1}{R_A^i} \sum_{\ell=1}^N  \int_{\R^d\setminus\{0\}}\big(1-e^{R_A^i h_t^{i,\ell}(x)}\big) \nu_t^{\ell,\alpha^\star}(\textrm{d}x) \bigg) \textrm{d}t  \\
& ~~~ - R_A^i Z_t^{i,:} \textrm{d}X^c_t + \frac{1}{R_A^i} \sum_{\ell=1}^N  \int_{\R^d\setminus\{0\}}\big(1-e^{R_A^i h_t^{i,\ell}(x)}\big) \mu_{X^\ell}(\textrm{d}x,\textrm{d}t) \\
& =  R_A^i \big( g^i(X_t,Z_t,\chi_t,\alpha_t^\star) - g^i(X_t,Z_t,\chi_t,\hat\alpha_t) \big) \textrm{d}t - R_A^i Z_t^{i,:}  \textrm{d}X_t^{c,\hat\alpha} \\
& ~~~ - \frac{1}{R_A^i} \sum_{\ell=1}^N  \int_{\R^d\setminus\{0\}}\big(1-e^{R_A^i h_t^{i,\ell}(x)}\big)\big(\mu_{X^\ell}(\textrm{d}x,\textrm{d}t) - \nu_t^{\ell,\hat\alpha}(\textrm{d}x) \textrm{d}t \big)
\end{align*}

It follows then that the process $U^{\alpha^i}$ is a $\P^{\alpha^i \otimes_i \alpha^{\star,-i}}$--local supermartingale of class (D) and hence a $\P^{\alpha^i \otimes_i \alpha^{\star,-i}}$--supermartingale. By the same argument, the process $U^{\alpha^{\star,i}}=\Ec\big(-R_A^i \int_0^\cdot Z^{i,:}_s \textrm{d}X_s^{c,\alpha^\star}\big) $ is a $\P^{\alpha^{\star}}$--martingale, so then
\[
 U_0^i(\alpha^{i},\alpha^{\star,-i},\xi^i,\chi^i) = \mathbb{E}^{\P^{\alpha^i \otimes_i \alpha^{\star,-i}}} \left[ U_T^{\alpha^i} \right] \leq U_0^{\alpha^i} =\Uc_A^i(Y_0^i) = U_0^{\alpha^{\star,i}} = \mathbb{E}^{\P^{\alpha^{\star}}}\left[ U_T^{\alpha^{\star,i}} \right] =  U_0^i(\alpha^{\star,i},\alpha^{\star,-i},\xi^i,\chi^i),
\]
which means that $\alpha^\star\in\text{NE}(\xi,\chi)$ and $V_0^i(\alpha^{\star,-i},\xi^i,\chi^i)=\Uc_A^i(Y_0^i)$.
\end{proof}

\bigskip
\begin{proof}[\bf Proof of Corollary \ref{cor-optimal-response-cert-equiv} ]
From Proposition \ref{prop-agent-bsde-cert-equiv}, by noting that $(Y^{y,Z,H,\chi},Z,H)$ is a solution to BSDE \eqref{eq:bsde-agent-cert-equiv}, we have that $a^{\star,Z,H,\chi}$ is a Nash equilibrium. Moreover, from part $(ii)$ of its proof we obtain the equality $V_0^i(\alpha^{\star,-i},\xi^i,\chi^i)=\Uc_A^i(Y_0^{y,Z,H,\chi,i})=\Uc_A^i(y^i)$.
\end{proof}

\section{Proofs for the principal's problem}  \label{app:proofs-principal}

\begin{proof}[\bf Proof of Proposition \ref{prop:unique-viscosity} ]
It is straightforward that assumptions $(2.1)-(2.5)$ in \cite{pham1998optimal} are satisfied. Then, the function $V$ is a viscosity solution of \eqref{eq:HJB-equation} by \cite[Theorem 3.1]{pham1998optimal} and has the mentioned continuity by \cite[Proposition 3.3]{pham1998optimal}. The uniqueness follows from the comparison result \cite[Theorem 4.1]{pham1998optimal}.
\end{proof}

\medskip
\begin{proof}[\bf Proof of Proposition \ref{prop:fsbde-viscosity}]
Note that $\psi(t,x,0)$ has linear growth because Assumption \ref{ass:pham-jumps} implies that $\phi(t,x,z,h,k)$ has linear growth, uniformly in $(z,h,k)$. We have that Assumptions (A.2i)-(A.2iV) and the ones in Section 1 in \cite{barles1997backward} are satisfied. Therefore, by \cite[Theorem 3.4]{barles1997backward} we have that $u$ is a viscosity solution to the I-PDE. From \cite[Proposition 2.5]{barles1997backward} we have that $u$ has polinomial growth. Under the additional assumptions on $L$ and $\psi$, it also follows from \cite[Proposition 2.5]{barles1997backward} that $u$ is uniformly continuous and bounded.
\end{proof}

\medskip
\begin{proof}[\bf Proof of Proposition \ref{prop-fujii-jumps} ]
We have that \cite[Assumptions 6.1-6.4]{fujii2018quadratic} are satisfied. Notice that \cite[Assumption 5.1]{fujii2018quadratic} holds trivially in our setting because the generator $\psi$ depends only on $(t,x,\zeta)$. Then, by \cite[Corollary 6.1]{fujii2018quadratic} we have that $u$ is continuous in $(t,x)$ and continuously differentiable with respect to $x$.
\end{proof}

\medskip
\begin{proof}[\bf Proof of Theorem \ref{thr-smoothness-and-existence} ]
$(i)$ Start by noting that the map $L$ is uniformly continuous from from Assumption \ref{ass:fujii-jumps} $(iv)-(v)$ and bounded from Assumption \ref{ass:fujii-jumps} $(iii)$. Next, the map $\psi(t,\cdot,\zeta)$ is uniformly continuous from Assumption \ref{ass:last-one} and bounded from Assumption \ref{ass:fujii-jumps} $(ii)$. From Proposition \ref{prop:fsbde-viscosity}, we know that $u$ is a viscosity solution to the I-PDE \eqref{eq:HJB-equation}. Moreover, by the second part of the result we have that $u$ is uniformly continuous and bounded which implies, due to Proposition \ref{prop:unique-viscosity}, that $u=V$. Finally, from Proposition \ref{prop-fujii-jumps} we have that $V$ is continuously differentiable with respect to $x$.

\medskip
$(ii)$ By following the same idea of %Propositions \ref{prop:unique-viscosity}, \ref{prop:fsbde-viscosity} and \ref{prop-fujii-jumps}
part $(i)$, we can prove that if the principal chooses controls $\theta=(\chi,Z,H)$ in problem \eqref{eq:reformulated-principal-cert-equiv-2} then she obtains utility $\tilde Y_0^{t,x,\theta}$, where  $(\tilde X^{t,x},\tilde Y^{t,x,\theta}, \tilde Z^{t,x,\theta}, \tilde H^{t,x,\theta})$ is the adapted solution to the FBSDE system
\[
\tilde X_s^{t,x} = x +\int_t^s \Sigma(\tilde X_r^{t,x}) \mathrm{d}W_r + \sum_{i=1}^N \int_t^s \int_{E\setminus\{0\}} I_{i}[ \beta^i(\tilde X_s^{t,x},e)] \mu_{J_i}(\mathrm{d}r,\mathrm{d}e),
\]
\[
\tilde Y_s^{t,x,\theta} = L(\tilde X_T^{t,x}) + \int_s^T \tilde\psi(r,\theta_r,\tilde X_r^{t,x},\tilde Z_r^{t,x,\theta}) \mathrm{d}r - \int_s^T \tilde Z_r^{t,x,\theta} \mathrm{d}W_r - \sum_{i=1}^N \int_s^T \int_{E\setminus\{0\}} \tilde H_r^{t,x,\theta}(e) \big(\mu_{J_i}(\mathrm{d}r,\mathrm{d}e) -F^i(\mathrm{d}e)\mathrm{d}r\big),
\]
with the generator
\[
\tilde\psi(t,k,z,h,x,\zeta) =  \phi(t,x,z,h,k) + \zeta \cdot b_t(x,\hat a^\star(t,x,z,h,k)).
\]
As the generator $\tilde\psi$ is Lipschitz in $\zeta$, as implied by Assumption \ref{ass:bbp-jumps} (iii), the comparison principle for BSDEs (see for instance \cite[Theorem 2.5]{royer2006backward}) gives us that $\tilde Y_0^{t,x,\theta}\leq \tilde Y_0^{t,x}$ for every $\theta=(\chi,Z,H)$ with $\chi\in\hat\Xi_2$ and $(Z,H)\in\hat\Vc^\chi$. The optimality of the control $\theta^\star=(\chi^\star,Z^\star,H^\star)$ follows by noting that $Y^{t,x}=Y^{t,x,\theta^\star}$.

\smallskip
To conclude, the form of the optimal contracts is a direct consequence of the reformulation of the principal's problem and the optimality of $(\chi^\star,Z^\star,H^\star)$.

\end{proof}

\end{document}